\documentclass[11pt]{amsart}

\usepackage[colorlinks=true, pdfstartview=FitV, linkcolor=blue,
citecolor=blue]{hyperref}

\usepackage{a4wide,xcolor}
\usepackage{amssymb,amsmath,amscd,mathdots}
\DeclareFontFamily{U}{mathx}{\hyphenchar\font45}
\DeclareFontShape{U}{mathx}{m}{n}{ <5> <6> <7> <8> <9> <10>
   <10.95> <12> <14.4> <17.28> <20.74> <24.88> mathx10 }{}
\DeclareSymbolFont{mathx}{U}{mathx}{m}{n}
\DeclareFontSubstitution{U}{mathx}{m}{n}
\DeclareMathAccent{\widecheck}{0}{mathx}{"71}

\DeclareMathAlphabet{\mymathbb}{U}{bbold}{m}{n}

\theoremstyle{plain}
\newtheorem{theorem}{Theorem}[section]
\newtheorem{prop}[theorem]{Proposition}
\newtheorem{lemma}[theorem]{Lemma}
\newtheorem{coro}[theorem]{Corollary}
\newtheorem{fact}[theorem]{Fact}

\theoremstyle{definition}
\newtheorem{definition}[theorem]{Definition}
\newtheorem{example}[theorem]{Example}
\newtheorem{remark}[theorem]{Remark}

\newcommand{\ts}{\hspace{0.5pt}}
\newcommand{\nts}{\hspace{-0.5pt}}

\newcommand{\CC}{\mathbb{C}\ts}
\newcommand{\EE}{\mathbb{E}\ts}
\newcommand{\RR}{\mathbb{R}\ts}
\newcommand{\ZZ}{{\ts \mathbb{Z}}}
\newcommand{\QQ}{{\ts \mathbb{Q}}}
\newcommand{\SSS}{\mathbb{S}}
\newcommand{\NN}{\mathbb{N}}

\newcommand{\cE}{\mathcal{E}}
\newcommand{\cF}{\mathcal{F}\ts}
\newcommand{\cH}{\mathcal{H}}
\newcommand{\cL}{\mathcal{L}}
\newcommand{\cO}{\mathcal{O}}
\newcommand{\cS}{\mathcal{S}}
\newcommand{\cY}{\mathcal{Y}}
\newcommand{\cZ}{\mathcal{Z}}
\newcommand{\cM}{\mathcal{M}}

\newcommand{\vG}{\varGamma}
\newcommand{\vL}{\varLambda}
\newcommand{\ii}{\mathrm{i}\ts}
\newcommand{\ee}{\mathrm{e}}
\newcommand{\dd}{\, \mathrm{d}}
\newcommand{\id}{\mathrm{id}}
\newcommand{\one}{\mymathbb{1}}
\newcommand{\nix}{\mymathbb{0}}

\newcommand{\exend}{\hfill $\Diamond$}
\newcommand{\defeq}{\mathrel{\mathop:}=}
\newcommand{\eqdef}{=\mathrel{\mathop:}}

\DeclareMathOperator{\dens}{dens}

\DeclareMathOperator{\Mat}{Mat}
\DeclareMathOperator{\supp}{supp}

\newcommand{\Cc}{C_{\mathsf{c}}}
\newcommand{\Cz}{C^{}_{0}}

\newcommand{\myfrac}[2]{\frac{\raisebox{-2pt}{$#1$}}
      {\raisebox{0.5pt}{$#2$}}}

\begin{document}

\title{On eigenmeasures under Fourier transform}

\author{Michael Baake}
\address{Fakult\"at f\"ur Mathematik,
         Universit\"at Bielefeld, \newline
\hspace*{\parindent}Postfach 100131, 33501 Bielefeld, Germany}
\email{$\{$mbaake,tspindel$\}$@math.uni-bielefeld.de }

\author{Timo Spindeler}

\author{Nicolae Strungaru}
\address{Department of Mathematical Sciences,
         MacEwan University, \newline
\hspace*{\parindent}10700 \ts 104 Avenue,
         Edmonton, AB, Canada T5J 4S2}
\email{strungarun@macewan.ca}



\begin{abstract}
  Several classes of tempered measures are characterised that are
  eigenmeasures of the Fourier transform, the latter viewed as a
  linear operator on (generally unbounded) Radon measures on
  $\RR^d$. In particular, we classify all periodic eigenmeasures on
  $\RR$, which gives an interesting connection with the discrete
  Fourier transform and its eigenvectors, as well as all eigenmeasures
  on $\RR$ with uniformly discrete support. An interesting subclass of
  the latter emerges from the classic cut and project method for
  aperiodic Meyer sets. Finally, we construct a large class of
  eigenmeasures with locally finite support that is not uniformly
  discrete and has large gaps around $0$.
\end{abstract}

\maketitle

\section{Introduction}

It is a well-known fact that the Fourier transform on $\RR$, more
precisely its extension to the Fourier--Plancherel transform on the
Hilbert space $\cH=L^2 (\RR)$, is a unitary operator of order~$4$,
with eigenvalues $\{ 1, \ii, -1, -\ii \}$.  A complete basis of
eigenfunctions in $\cH$ can be given in terms of the classic Hermite
functions \cite[Thm.~57]{T}. In physics, these functions are well
known as the eigenfunctions of the classic harmonic oscillator in
quantum mechanics, which is a self-adjoint operator on $\cH$ that
commutes with the Fourier transform; see \cite[Sec.~34]{Dirac} as well
as \cite{DK}, and Section~\ref{sec:prelim} for details and our
conventions.

The corresponding statement for $L^2 (\RR^d)$ can easily be derived
from here, and is helpful in many applications of the Fourier
transform \cite{Wiener, DK}, both in mathematics and in physics. More
generally, eigenfunctions of Fourier cosine and sine transforms as
well as Hankel, Mellin and various other integral transforms have been
studied \cite[Ch.~IX]{T}, in the setting of different function
spaces. Clearly, these results can be reformulated in the realm of
\emph{finite}, absolutely continuous measures, and can be
extended by taking limits in suitable topologies. This point of view
obviously harvests the self-duality of the locally compact Abelian
group $\RR^d$, while the important analogues for mutually dual pairs
of groups, such as the integer lattice $\ZZ^d$ and the corresponding
torus $\mathbb{T}^d$, lead to the theory of Fourier series and its
applications \cite{Rudin,DK,Katz}, as well as to the discrete Fourier
transform.

There is one famous and well-known extension to measures as
follows. If $\delta_x$ denotes the normalised Dirac (or point)
measure at $x$, the lattice Dirac comb
$\delta^{}_{\ZZ^d} \defeq \sum_{m\in\ZZ^d} \delta^{}_{m}$ satisfies
\begin{equation}\label{eq:PSF-1}
  \widehat{\delta^{}_{\ZZ^d}} \, = \, \delta^{}_{\ZZ^d} \ts ,
\end{equation}
where $\widehat{\cdot}$ denotes Fourier transform.
{This formula means that
\begin{equation}\label{eq:PSF-s}
  \sum^{}_{x \in \ZZ^d} f(x)  \, =
  \sum^{}_{y \in \ZZ^d} \widehat{f}(y)
\end{equation}
holds for all Schwartz functions $f$, which} is nothing but the
Poisson summation formula (PSF) for $\ZZ^d$, written in terms of a
lattice Dirac comb; see \cite{Cordoba,BM} as well as
\cite[Sec.~9.2]{TAO1} and references therein. {In fact,
  \eqref{eq:PSF-s} holds for the Feichtinger algebra $S_0 (\RR^d)$
  \cite[Thm.~4.2]{Fei06}, which gives a stronger version of the PSF
  because $S_0 (\RR^d)$ contains all Schwarz functions
  \cite[Thm.~9]{Fei81}. The PSF says that} $\delta^{}_{\ZZ^d}$ is an
eigenmeasure of the Fourier transform with eigenvalue $1$.  More
precisely, in contrast to the function case mentioned above, it is an
example of an \emph{unbounded}, but still translation-bounded, Radon
eigenmeasure with uniformly discrete support.

In mathematics, the PSF (for $d=1$) is related with the functional
equation of Riemann's zeta function, via a Mellin transform, while its
relevance in physics comes from its role in the understanding of the
diffraction theory of crystals. By the latter, we mean
lattice-periodic arrangements of a fundamental motif, which shows a
pure point (or pure Bragg) diffraction spectrum in scattering
experiments with X-rays or other particle beams. The spectrum can then
be understood, both qualitatively and quantitatively, via the Fourier
transform of lattice-periodic measures of the form
$\varrho * \delta^{}_{\nts\vG}$ with a finite motif $\varrho$ and a
general lattice $\vG$ in $\RR^d$; see \cite[Sec.~9.2]{TAO1} for a
detailed exposition. The crucial point here is that the lattice can be
extracted from the support of the spectrum, while details of the
motif need the solution of the difficult inverse problem of
crystallographic structure analysis \cite{Cow}.

More generally, and perhaps more importantly as well, the PSF also
underlies the pure point nature of the diffraction spectra of perfect
quasicrystals \cite{Shecht,Hof,Jeff-rev,TAO1}, which is an important
branch in the theory of aperiodic order. This property emerges from
the description of such quasicrystals as the partial projection of a
higher-dimensional lattice (or its generalisation in the setting of
locally compact Abelian groups \cite{Meyer,Moo00}), where the
embedding lattice and its PSF drives the spectral nature. While the
original approach did not use the PSF, it was suspected to be the key
ingredient \cite{Jeff-rev}, and later established as such
\cite{TAO1,RS}. Crystals and quasicrystals together form an
interesting class of structures within the larger universe of almost
periodic measures; see \cite{MoSt} for a recent survey. Moreover, one
can embed crystals and quasicrystals into the larger class of
modulated (quasi-)crystals, whose spectral structure is still
determined by an underlying PSF. A unified treatment of this entire
class has recently been achieved in \cite{many}.

In this context, it is thus a natural question whether other unbounded
eigenmeasures of the Fourier transform exist, which eigenvalues occur,
and what they might mean. An interesting example, based on work by
Guinand \cite{Gui}, is discussed in \cite{Meyer}. In our terminology,
it is a {tempered} eigenmeasure with eigenvalue $\ii$, but the measure
does not have a uniformly discrete support.  This is a {particularly
  striking} example of what is now known as a `crystalline' measure
\cite{Meyer,LO3,Kol}, and first results indicate that the underlying
measure class deserves further study. In particular, it has an
interesting overlap with the class of mathematical quasicrystals,
which includes the measure $\delta^{}_{\ZZ^d}$ mentioned earlier. This
observation, in conjunction with the general questions put forward in
\cite{Jeff-rev} and the potential relevance to quasicrytals, inspired
us to take a systematic look at translation-bounded eigenmeasures, in
particular with uniformly discrete support, though we also look
beyond. The connection between the diffraction and the dynamical
spectrum for translation bounded measures, see \cite{BL,DM} and
references therein, has potential implications of our findings for the
spectral theory of dynamical systems, but we leave this to future
work. The paper is organised as follows.  \smallskip

After some preliminaries in Section~\ref{sec:prelim}, we determine the
possible eigenvalues and some first solutions, in the setting of Radon
measures, in Section~\ref{sec:exists}. Here, we also derive a simple
characterisation of eigenmeasures in terms of $4$-cycles under the
Fourier transform (Theorem~\ref{thm:decomp-eigenspaces}).  While this
is classic material, and can thus be viewed as a condensed summary, we
are paying more attention to the difference between distributions and
measures, which is necessary outside the class of positive measures
\cite{BS}.  Then, for $d=1$, we consider lattice-periodic solutions
(Section~\ref{sec:cryst}) more closely, which are related with
eigenvectors of the \emph{discrete Fourier transform} (DFT), that is,
the Fourier transform on the finite cyclic group $C_n$, which is also
self-dual. In this periodic case, all eigenmeasures are automatically
pure point measures that are supported in a lattice
(Theorem~\ref{thm:lattice}).

Next, in Section~\ref{sec:shadows}, we consider solutions that
implicitly emerge from the cut and project method of aperiodic order
(Theorem~\ref{thm:aper-meas} and Corollary~\ref{coro:aper}).  Then, we
characterise the general class of eigenmeasures with uniformly
discrete support in Section~\ref{sec:discrete}, leading to
Theorem~\ref{thm:main}, which builds on recent results from
\cite{LO1,LO2}.  Finally, we construct lattice-based eigenmeasures
with a large gap around $0$, which is possible via the connection with
the DFT. The crucial point here is that this construction, via
suitable linear combinations, paves the way to Fourier eigenmeasures
with locally finite, but no longer uniformly discrete, support
(Theorem~\ref{thm:Meyer1} and Corollary~\ref{coro:ev-gap}).  We close
with a brief outlook, and provide an Appendix with a streamlined and
tailored approach to the DFT.

\section{Preliminaries}\label{sec:prelim}

The \emph{Fourier transform} of a function $g\in L^{1} (\RR)$ is
defined by
\[
    \widehat{g} \ts (y) \, = \int_{\RR} \ee^{-2 \pi \ii x y}
    \ts g(x) \dd x \ts ,
\]
which is bounded and continuous, while the inverse transform is given
by $\widecheck{g} (y) = \widehat{g} (-y)$. If
$\widecheck{g} \in L^{1} (\RR)$, one has
$\widehat{\nts\widecheck{g}\ts} = g$.  For the (complex) Hilbert space
$\cH = L^{2} (\RR)$ with inner product
\[
    \langle \ts g \ts | f \ts \rangle  \, \defeq \,
    \int_{\RR} \overline{g (x)} \ts f (x) \dd x
\]
and norm $\| f \|^{}_{2} \defeq \sqrt{\langle f \ts | f \rangle}$, the
Fourier transform on $L^{2}(\RR) \cap L^{1} (\RR)$ has a unique
extension to a unitary operator $\cF$ on $L^{2} (\RR) = \cH$, which is
often called the \emph{Fourier--Plancherel transform}; see
\cite[Thm.~2.5]{BF}. The corresponding statements hold for $\RR^d$ as
well. $\cF$ satisfies the relation $\cF^2 = I$, where $I$ is the
involution defined by $\bigl( I g \bigr) (x) = g (-x)$.  Consequently,
one also has $\cF^4 = \id$, which implies that any eigenvalue of $\cF$
must be a fourth root of unity.

It is well known that $\cH$ possesses an ON-basis of eigenfunctions,
which can be given via the Hermite functions.  They naturally also
appear as the eigenfunctions of the harmonic oscillator in quantum
mechanics, which is a self-adjoint operator that acts on $\cH$ and
commutes with $\cF$; compare \cite[Sec.~34]{Dirac} and \cite{DK}.
Concretely, for $n\in \NN_0$, one can employ the Hermite polynomials
$H_n$ defined by the recursion
\[
    H_{n+1} (x) \, = \, 2 \ts x \ts H_n (x) - 2 \ts n \ts H_{n-1} (x)
\]
for $n\in\NN$ together with $H^{}_{0} (x) = 1$ and
$H^{}_{1} (x) = 2 x$.  Then, the (normalised) Hermite functions
\[
  h_{n} (x) \, \defeq \, \biggl( \frac{\sqrt{2}}{2^n n!}
  \biggr)^{\! 1/2}
    H_n \bigl( \sqrt{2 \pi} \ts x \bigr) \, \ee^{- \pi x^2}
\]
satisfy the orthonormality relations
$\langle h_m \ts | \ts h_n \rangle = \delta_{m,n}$ for
$m,n \in \NN_0$.

Moreover, for $n\geqslant 0$, one has
\[
  \cF (h^{}_{n}) \, = \,
  \widehat{h^{}_{n}} \, = \, (-\ii)^n h^{}_{n} \ts ,
\]
which means that the Hermite functions are eigenfunctions of $\cF$;
see \cite[Rem.~8.3]{TAO1} as well as \cite[\S 6]{Wiener} or
\cite[Sec.~2.5]{DK}.  In Dirac's intuitive bra-c-ket notation
\cite[Sec.~14]{Dirac}, one thus obtains the spectral theorem for the
unitary operator $\cF$ as
\[
    \cF \, = \sum_{n=0}^{\infty} | \ts h_n \ts \rangle (-\ii)^n
    \langle \ts h_n \ts | \ts ,
\]
which also entails the relation
$\cH = \cH^{}_{0} \oplus \cH^{}_{1} \oplus \cH^{}_{2} \oplus
\cH^{}_{3}$, with $\cH^{}_{\ell} = \text{span}^{}_{\CC}
\{ h^{}_{4m+\ell} : m\in \NN_{0} \}$.

Next, let $xy$ denote the inner product of $x,y\in\RR^d$ and let $\mu$
be a finite measure on $\RR^d$. Then, its Fourier transform
\cite{Rudin} is the continuous function on $\RR^d$ given by
\[
   \widehat{\mu} (y) \, \defeq \int_{\RR^d}
   \ee^{-2 \pi \ii x y} \dd \mu (x) \ts .
\]
As the Fourier transform of a finite measure is a continuous function,
nothing much of interest happens beyond what we summarised
above. All this can be understood by taking limits of absolutely
continuous measures (on the basis of sums of Hermite functions
as Radon--Nikodym densities) in suitable topologies.
Consequently, we proceed with the extension to \emph{unbounded}
measures, where the very notion of Fourier transformability is more
delicate. The simplest approach works via distribution theory, which
we shall adopt here. From now on, unless specified otherwise, the term
\emph{measure} will always mean a Radon measure, the latter viewed
as a linear functional on $\Cc (\RR^d)$.

Since $\Cc^{\infty} (\RR^d) \subset \Cc (\RR^d)$, every measure defines
a distribution. Now, a measure on $\RR^d$ is \emph{tempered}
if this distribution is tempered, that is, a continuous linear
functional on Schwartz space, $\cS (\RR^d)$.
A tempered measure $\mu$ is called \emph{Fourier
  transformable}, or \emph{transformable} for short, if it is
transformable as a distribution and if the transform is again a
measure. The distributional transform of $\mu$ will interchangeably be
denoted by $\cF (\mu)$ and by $\widehat{\mu}$. When the measure $\mu$
is transformable in this sense, its transform $\nu = \widehat{\mu}$ is
a tempered measure, and hence transformable as a distribution, with
$\widehat{\nu} = \cF^2 (\mu) = I \nts . \ts \mu$, where the latter
means the push-forward of the space inversion defined by $I (x) = -x$,
to be discussed in more detail later. As a consequence,
$\widehat{\nu}$ is again a measure, and we have the following useful
property.

\begin{fact}\label{fact:forever}
  If a tempered measure is transformable, it is also multiply
  transformable. \qed
\end{fact}

Since $\RR^d$ is a special case of a locally compact Abelian group,
let us also mention the classic approach to the transformability of a
measure, which completely avoids the use of distribution theory, and
is not restricted to tempered measures.

\begin{definition}
  A Radon measure $\mu$ is called \emph{transformable in the strict
    sense}, or \emph{strictly transformable} for short, if it is
  transformable as a measure, in the sense of \cite{ARMA1,BF}.  This
  means that there exists some measure $\widehat{\mu}$ on $\RR^d$ with
  the property that, for all $\varphi \in \Cc(\RR^d)$ with
  $|\widecheck{\varphi}\ts | \in L^1(\RR^d)$, we have
  $ \mu(\varphi)\,=\,\widehat{\mu}( \widecheck{\varphi})$.
\end{definition}

We refer to \cite{MoSt,NS11} for the systematic exposition of recent
developments and results.  Note that a transformable, tempered measure
need not be strictly transformable. By slight abuse of notation, we
shall use $\cF$ for both versions, as the context will always be
clear.  As the domain of two successive applications of $\cF$ in the
strict sense is now generally smaller than that of $\cF$, one also
needs the notion of double transformability here. A measure $\mu$ is
called \emph{doubly} (or \emph{twice}) transformable in the strict
sense if $\mu$ is strictly transformable as a measure and its
transform, $\cF(\mu)$, is again transformable in the strict sense.

Recall that a measure $\mu$ on $\RR^d$ is called \emph{translation
  bounded} if $\sup_{t\in\RR^d} \lvert \mu \rvert (t+K) < \infty$
holds for some (and then any) compact set $K\subset\RR^d$ with
non-empty interior, where $\lvert \mu\rvert$ is the total variation of
$\mu$. This notion is crucial in many ways. In particular, every
translation-bounded measure is tempered \cite[Thm.~7.1]{ARMA1}.
Further, strict transformability of tempered measures can
most easily be decided via an application of \cite[Thm.~5.2]{Nicu} as
follows; see also \cite{ST}.

\begin{fact}\label{fact:transformable}
  A tempered measure\/ $\mu$ on\/ $\RR^d$ is strictly transformable as
  a measure if and only if it is transformable as a tempered
  distribution such that\/ $\widehat{\mu}$ is a translation-bounded
  measure.  Moreover, the two transforms agree in this case.  \qed
\end{fact}

One important consequence of Fact~\ref{fact:transformable} is that a
strictly transformable eigenmeasure is automatically translation
bounded. However, this type of approach via distributions is only
available over $\RR^d$.  A characterisation of transformable measures
on locally compact Abelian groups can be found (without proof) in
\cite[Thm. C1]{Fei}.  For some subtleties around tempered measures,
once one goes beyond the class of positive measures, we also refer to
\cite{BS}.

\begin{remark}
  Recall that a measure $\mu$ is \emph{slowly increasing} if
\[
    \int_{\RR^d} \frac{\dd | \mu | \ts (x) }{1+|P(x)|} \, < \, \infty \ts
\]
holds for some polynomial $P$.
By \cite[p.~242]{Schw}, a positive measure is tempered if and only if
it is slowly increasing. A slowly increasing measure is always tempered,
but the converse is generally not true; compare \cite[Prop.~7.1]{ARMA1}.

ome authors \cite{Kab} call a Radon measure $\mu$ \emph{tempered} when
\[
   \int_{\RR^d} |f(x)| \dd |\mu|(x) \, < \, \infty
\]
holds for all $f \in \cS(\RR^d)$ and when
$f \mapsto \int_{\RR^d} f(x) \dd \mu(x)$ defines a tempered
distribution. Note that a measure is tempered in this sense if and
only if it is slowly increasing \cite[Thm.~2.7]{BS}.  The somewhat
subtle but relevant issue with this definition is that the example
from \cite{ARMA1} mentioned above provides an example of a Radon
measure $\mu$ that is not slowly increasing, but whose restriction to
$\Cc^\infty(\RR^d)$ can be extended to a tempered distribution.  Let
us note in passing that, for this measure, there exists
$f \in \cS(\RR^d)$ such that
\[
   \int_{\RR^d} |f(x)| \dd |\mu|(x) \, = \, \infty \ts ;
\]
see \cite{BS} for details.
\exend
\end{remark}

By \cite[Thm.~3.12]{RS}, a strictly transformable measure $\mu$ on
$\RR^d$ is doubly transformable in the strict sense if and only if
$\mu$ is translation bounded. In this case, by
\cite[Thm.~4.9.28]{MoSt}, one actually gets
$\cF^2 (\mu) = I \nts . \ts \mu$ as before.

\begin{definition}\label{def:feigen}
  A non-trivial tempered measure\/ $\mu$ on\/ $\RR^d$ is called an
  \emph{eigenmeasure} of\/ $\cF$, or an eigenmeasure under Fourier
  transform, if it is transformable together with\/
  $\cF (\mu) = \lambda \mu$ for some\/ $\lambda \in \CC$, which is the
  corresponding \emph{eigenvalue}.

  Further, we call an eigenmeasure\/ $\mu$ \emph{strict} when it is
  transformable in the strict sense and satisfies\/
  $\cF (\mu) = \lambda \mu$ in the sense of measures.
\end{definition}

When the support of a measure $\mu$ has special discreteness
properties, one can profitably employ the following equivalence.

\begin{lemma}\label{lem:tb}
  If\/ $\mu\ne 0$ is a tempered, pure point measure on\/ $\RR^d$ with
  uniformly discrete support, the following properties are equivalent.
\begin{enumerate}\itemsep=2pt
\item The measure\/ $\mu$ is strictly transformable, and\/
  $\widehat{\mu} = \lambda \mu$ holds as an equation of measures for
  some\/ $0\ne\lambda\in \CC$.
\item The measure\/ $\mu$ satisfies\/ $\widehat{\mu} = \lambda \mu$ in
  the distributional sense for some\/ $0\ne\lambda \in \CC$.
\end{enumerate}
\end{lemma}

\begin{proof}
  Let us first note that $\lambda = 0$ implies $\mu = 0$, so we must
  indeed have $\lambda \ne 0$.

  To verify (1) $\Rightarrow$ (2), we simply observe that strict
  transformability of $\mu$ implies its (distributional)
  transformability by Fact~\ref{fact:transformable}, and the two
  transforms agree.

  Conversely, \cite[Lemma~6.3]{BST} implies that $\widehat{\mu}$ is
  translation bounded, and the claim follows again from
  Fact~\ref{fact:transformable}.
\end{proof}

These characterisations will cover essentially all situations we
encounter below. Let us first show that the eigenmeasure notion is a
consistent one, and determine the possible eigenvalues.

\section{Eigenvalues and structure of solutions}\label{sec:exists}

Recall that $\cS (\RR^d)$ denotes Schwartz space \cite{Schw}, and that
$I$ is the reflection in $0$, hence \mbox{$I(x) = -x$}. As above, we
let $I$ also act on functions, via $I (g) \defeq g \circ
I$. Accordingly, when $\mu$ is a measure, we define the push-forward
$I \nts . \ts \mu$ by
$\bigl( I \nts . \ts \mu \bigr) (g) = \mu ( g \circ I)$.

Let $\mu\ne 0$ be a tempered measure, so $\mu \in \cS' (\RR^d)$, where
the latter denotes the space of tempered distributions on $\RR^d$. If
$\mu$ is an eigenmeasure of $\cF$, hence $\cF (\mu) = \lambda \mu$, we
get $\cF^4 (\mu) = \lambda^4 \mu$ by repetition and
Fact~\ref{fact:forever}.  On the other hand, given any
$\varphi\in \cS (\RR^d)$ and observing the relation
$\cF^2 (\varphi) = \varphi \circ I$, one has
\[
    \cF^2 (\mu) (\varphi) \, = \, \mu \bigl(
    \cF^2 (\varphi) \bigr) \, = \, \mu ( \varphi \circ I )
    \, = \, \bigl( I \nts . \ts \mu \bigr) (\varphi) \ts ,
\]
hence $\cF^2 (\mu) = I\nts . \ts \mu$ and thus $\cF^4 (\mu) = \mu$ due
to $I^2 = \id$. This implies $\lambda^4 = 1$ as in the case of
functions, so $\lambda$ must be a fourth root of unity. Note that we
also get $I\nts . \ts \mu = \cF^2 (\mu) = \lambda^2 \mu$, for any of
the possible eigenvalues. This has an immediate consequence as
follows.

\begin{fact}\label{fact:symm}
  Let\/ $\mu$ be a tempered measure on\/ $\RR^d$ that is an
  eigenmeasure of\/ $\cF$, so\/ $\widehat{\mu} = \lambda \mu \ne
  0$. Then, $\lambda \in \{ 1, \ii, -1, -\ii \}$, and one has\/
  $I \nts . \ts \mu = \lambda^2 \mu$. In particular, the support of\/
  $\mu$ is symmetric, $\supp (\mu) = - \supp (\mu)$, and\/ $\mu$
  satisfies\/
  $\mu \bigl( \{ -x \} \bigr) = \lambda^2 \ts \mu \bigl( \{ x \}
  \bigr)$ for all\/ $x\in\RR^d$.  \qed
\end{fact}

Let us verify that all $\lambda \in \{ 1, \ii, -1, -\ii \}$ actually
occur. To this end, we use $\cF^4=\id$ and consider measures
$\nu^{}_{m} = \mu + \ii^m \cF (\mu) + \ii^{2m} \cF^2 (\mu) + \ii^{3m}
\cF^3 (\mu)$, for $0\leqslant m \leqslant 3$.  Then, one has
\begin{equation}\label{eq:trick}
      \cF (\nu^{}_{m} ) \, = \, (-\ii)^m \ts \nu^{}_{m} \ts ,
\end{equation}
which establishes the existence of eigenmeasures for all possible
$\lambda \in \{ 1, \ii, -1, -\ii \}$, provided we start from a
transformable, tempered measure $\mu$ such that the $\nu^{}_{m}$ with
$0\leqslant m \leqslant 3$ are non-trivial. On $\RR$, this can be
achieved with $\mu = \delta^{}_{\alpha} \nts * \delta^{}_{\ZZ}$ for
any $\alpha \not\in \QQ$.

Let us continue with a slightly more complex example that will give
additional insight.  To this end, we set $d=1$ and start with an
arbitrary real eigenmeasure for $\lambda=1$, so
$\widehat{\mu} = \mu = \overline{\mu}$, such as
$\mu = \delta^{}_{\ZZ}$ from the PSF in \eqref{eq:PSF-1}. Now, let
$\chi^{}_{s} \! : \, \RR \xrightarrow{\quad} \SSS^1$ be a character,
so $\chi^{}_{s} (x) = \ee^{2 \pi \ii s x}$ for some fixed $s \in
\RR$. Then, for $t\in\RR$, a simple calculation shows that
\begin{equation}\label{eq:char-1}
  \delta^{}_{t} * \bigl( \chi^{}_{s} \ts \delta^{}_{\ZZ} \bigr)
    \, = \, \overline{\chi^{}_{s} (t)} \,
    \chi^{}_{s} \cdot \bigl( \delta^{}_{t} \nts * \delta^{}_{\ZZ} \bigr)
     \, = \, \overline{\chi^{}_{s} (t)} \,
    \bigl( \chi^{}_{s} \ts \delta^{}_{\ZZ + t} \bigr)  .
\end{equation}
This can be done more generally for any real eigenmeasure, with the
following consequence.

\begin{lemma}\label{lem:modif}
  Let\/ $\chi^{}_{s}$ be a character on\/ $\RR$, and let\/ $\mu$ be a
  transformable, tempered measure on\/ $\RR$ that satisfies\/
  $\widehat{\mu} = \mu = \overline{\mu}$. Then, the measure\/
  $\omega^{}_{s} \defeq \chi^{}_{s} \cdot (\delta^{}_{s} \nts * \mu)$
  is a tempered measure that is transformable as well, with
\[
    \widehat{\omega^{}_{s}} \, = \,
    \ee^{2 \pi \ii s^2}\ts \overline{\omega^{}_{s}} \ts .
\]
For $\mu = \omega^{}_{0}$, this relation reduces to\/
$\widehat{\mu} = \mu$.

Moreover, if\/ $\mu$ is transformable in the strict sense, then so is
$\omega^{}_s$, and both\/ $\mu$ and\/ $\omega^{}_{s}$ are
translation-bounded measures.
\end{lemma}

\begin{proof}
  Under our assumptions, $\omega^{}_{s}$ clearly is a tempered
  measure, while a simple calculation with the convolution theorem
  shows that the distribution $\widehat{\omega^{}_{s}}$ is indeed a
  measure, so $\omega^{}_{s}$ is transformable by definition.

  With $\mu$, also the shifted measure $\delta^{}_{s} \nts * \mu$ is
  real. Observing $\widecheck{\ts\delta^{}_s} = \chi^{}_s $, we can simply
  calculate
\[
   \cF (\omega^{}_{s}) \, = \, \widehat{\chi^{}_{s}} *
   \bigl( \chi^{}_{-s} \, \widehat{\mu} \bigr) \, = \,
   \delta^{}_{s} \nts * (\overline{\chi^{}_{s}} \, \mu )
   \, = \, \chi^{}_{s} (s)
   \bigl( \overline{\chi^{}_{s}} \cdot (\delta^{}_{s}
   \nts * \mu ) \bigr) \, = \, \ee^{2 \pi \ii s^2} \,
   \overline{\omega^{}_{s}} \ts ,
\]
where the third step is the obvious generalisation of
\eqref{eq:char-1}.  This establishes the main claim, while $s = 0$
clearly gives the eigenmeasure relation stated.

When $\mu$ is strictly transformable, $\mu= \widehat{\mu}$
is translation bounded by Fact~\ref{fact:transformable}. It follows
immediately that $\omega^{}_{s}$ is translation bounded as well.
\end{proof}

With the $\omega^{}_{s}$ from Lemma~\ref{lem:modif}, for generic $s$,
we get the cycle
\begin{equation}\label{eq:cycle}
    \omega^{}_{s} \, \xrightarrow{\,\cF\,} \,
    \ee^{2 \pi \ii s^2} \, \overline{\omega^{}_{s}}
     \, \xrightarrow{\,\cF\,} \,  I \nts\nts . \ts\ts \omega^{}_{s}
      \, \xrightarrow{\,\cF\,} \, \ee^{2 \pi \ii s^2} \,
      I \nts\nts . \ts\ts \overline{\omega^{}_{s}}
       \, \xrightarrow{\,\cF\,} \,  \omega^{}_{s} \ts ,
\end{equation}
as can be checked by an explicit calculation that we leave to the
interested reader. Unless $s$ takes special values, the four measures
are distinct, and thus form a $4$-cycle. Excluding a few more values
for $s$, the measures $\nu^{}_{m}$ from \eqref{eq:trick} with
$\mu = \omega^{}_{s}$ are non-trivial, so we see a plethora of
eigenmeasures of $\cF$, for all possible eigenvalues.

{Next, for each $\lambda \in \{ \pm 1, \pm\ts \ii \}$, we define an
operator
$P^{}_{\lambda} \colon \cS' (\RR^d) \xrightarrow{\quad} \cS' (\RR^d)$
via
\[
  P^{}_{\lambda} (\mu) \, \defeq \, \myfrac{1}{4}
  \sum_{j=0}^{3} \lambda^{-j} \ts \cF^{j} (\mu) \ts .
\]
Its properties can be summarised as follows.

\begin{prop}\label{prop:plambda}
The operators\/ $P^{}_{\lambda}$ with\/ $\lambda \in \{ 1, \ii, -1, -\ii\}$
satisfy the following properties.
\begin{enumerate}\itemsep=2pt
\item One has\/ $P^{}_1+P^{}_{\ts\ii}+P^{}_{-1}+P^{}_{-\ii}= \id$.
\item For any\/ $\lambda, \lambda' \in \{ \pm 1, \pm\ts \ii \}$, one
  has\/ $P^{}_{\lambda}\nts \circ P^{}_{\lambda'} = \delta^{}_{\lambda,\lambda'}
  \ts P^{}_{\lambda}$ and\/
  $\cF \nts \circ P^{}_\lambda = \lambda\ts P^{}_\lambda$.
\item A subspace\/ $X \nts\subseteq\cS' (\RR^d)$ is
  $\cF\nts$-invariant if and only if\/
  $P^{}_{\lambda} (X) \subseteq X$ for all\/
  $\lambda \in \{ \pm 1, \pm \ts \ii \}$.
\item Let\/ $X \subseteq \cS' (\RR^d)$ be\/ $\cF\nts$-invariant and
  let\/ $\lambda \in \{ \pm 1, \pm\ts \ii \}$. Then, the restriction
  of\/ $P^{}_{\lambda}$ to\/ $X$ is a projection operator onto the
  eigenspace
\[
    \EE^{}_{\lambda}(X) \, \defeq \,
    \{ \omega \in X : \cF(\omega) = \lambda \ts \omega \} \ts .
\]
In particular, one has\/ $ X  =  \EE^{}_{1} (X)
  \oplus \EE^{}_{\ts\ii} (X) \oplus
  \EE^{}_{-1} (X) \oplus
  \EE^{}_{-\ii} (X) $.
\end{enumerate}
\end{prop}

\begin{proof}
Note first that, for all $\omega \in \cS' (\RR^d)$, we have
\[
  \sum_{k=0}^{3} P^{}_{\ii^k}(\omega)  \, = \,
  \myfrac{1}{4} \sum_{k=0}^3 \sum_{j=0}^{3} \ii^{-kj} \ts
  \cF^{j} (\omega) \, = \, \myfrac{1}{4}  \sum_{j=0}^{3}
  \biggl( \sum_{k=0}^{3} \ii^{-kj} \biggr) \ts \cF^{j} (\omega) \ts .
\]
Now, the bracketed sum vanishes unless  $j=0$, which gives
\begin{equation}\label{eq:sumE_i}
   \sum_{k=0}^{3} P^{}_{\ii^k} (\omega) \, = \, \omega
\end{equation}
and establishes Property (1).

For Property (2), for all $\lambda \in \{ \pm 1, \pm \ts \ii \}$
  and all $\omega \in \cS' (\RR^d)$, we have
\begin{align*}
  \cF \bigl( P^{}_{\lambda} (\omega)  \bigr) \,
  & = \, \cF \biggl( \myfrac{1}{4}
  \sum_{j=0}^{3} \lambda^{-j} \ts \cF^{j} (\omega) \biggr) \, = \,
  \frac{1}{4} \sum_{j=0}^{3} \lambda^{-j} \ts \cF^{j+1} (\omega) \\[1mm]
  & = \, \myfrac{\lambda}{4} \sum_{j=0}^{3} \lambda^{-j-1}
   \ts \cF^{j+1} (\omega) \, = \, \lambda  P^{}_{\lambda} (\omega) \ts ,
\end{align*}
with the last equality following from $\cF^4=\mbox{id}$ and
$\lambda^4=1$. Furthermore,
\[
  P^{}_{\lambda} \bigl( P^{}_{\lambda'} (\omega) \bigr) \,
   = \, \biggl( \myfrac{1}{4}
  \sum_{j=0}^{3} \lambda^{-j} \ts \cF^{j}
  \bigl( P^{}_{\lambda'} (\omega) \bigr) \biggr) \, = \,
  \myfrac{1}{4} \sum_{j=0}^{3} \lambda^{-j}
  \bigl( {\lambda'}^{j} P^{}_{\lambda'} (\omega) \bigr)
  \, = \, \myfrac{1}{4} \sum_{j=0}^{3}
  \bigl( \lambda^{-1}\ts \lambda' \bigr)^{j} \,
  P^{}_{\lambda'} (\omega)  \ts .
\]
Since $\lambda^{-1}\ts \lambda'$ is a fourth root of unity, we
have $\sum_{j=0}^{3} \bigl( \lambda^{-1}\ts \lambda' \bigr)^{j} = 4 \ts
\delta^{}_{\lambda,\lambda'}$, which proves this part.

The direct implication in Property (3) follows immediately from the
definition of $P^{}_\lambda$, while the converse follows from
Properties (1) and (2).

Finally, for Property (4), observe first that we have
  $P^{}_{\lambda} (X) \subseteq \EE^{}_{\lambda} (X)$ from
  Property (2).  Moreover, for all
  $\omega \in \EE^{}_{\lambda}(X)$, we have
\[
  P^{}_{\lambda}( \omega) \,
   = \, \myfrac{1}{4}
  \sum_{j=0}^{3} \lambda^{-j} \ts \cF^{j} (\omega)  \, = \,
  \myfrac{1}{4} \sum_{j=0}^{3} \lambda^{-j}
  \bigl( \lambda^j \ts \omega \bigr)
  \, = \, \omega  \ts ,
\]
so $P^{}_{\lambda} \colon X \! \xrightarrow{\quad}
\EE^{}_{\lambda}(X)$ is onto. The claim now follows from
Properties (2) and (3).
\end{proof}

\begin{remark}
  Using the work of Simon \cite{Sim}, see also \cite[Thm.~1]{CFK}, one
  can decompose $ P^{}_{\lambda}(\omega)$ as a series of Hermite
  functions, which are eigenfunctions for $\lambda$. Below, we are
  primarily interested in the case where both $\omega$ and
  $\widehat{\omega}$ are tempered measures with locally finite
  support. In this case, the Hermite function expansion does not seem
  to give further insight.  \exend
\end{remark}

Now, let us denote by $\cM_{\text{ttm}}(\RR^d)$ the space of
transformable tempered measures on $\RR^d$, which by definition is
$\cF\nts$-invariant. Therefore, we get the following result.

\begin{theorem}\label{thm:decomp-eigenspaces}
  For each\/ $\lambda \in \{ \pm 1, \pm \ts \ii \}$, the mapping\/
  $P^{}_{\lambda}$ induces a projection operator from\/
  $\cM_{\mathrm{ttm}} (\RR^d)$ onto the eigenspace\/
  $\EE_{\lambda} \bigl(\cM_{\mathrm{ttm}}
  (\RR^d) \bigr)$, with\/  $P^{}_1+P^{}_{\ts\ii}+P^{}_{-1}
  +P^{}_{-\ii}= \id$. In particular,
  $\cM_{\mathrm{ttm}} (\RR^d) =
  \EE^{}_{1} \bigl(\cM_{\mathrm{ttm}} (\RR^d)\bigr)
  \oplus \EE^{}_{\ts\ii} \bigl(\cM_{\mathrm{ttm}}
  (\RR^d)\bigr) \oplus \EE^{}_{-1}
  \bigl(\cM_{\mathrm{ttm}} (\RR^d) \bigr) \oplus
  \EE^{}_{-\ii} \bigl(\cM_{\mathrm{ttm}}
  (\RR^d)\bigr)$. \qed
\end{theorem}}

This shows that there is an abundance of eigenmeasures, which makes
a meaningful classification difficult unless certain subclasses are
specified. In this framework, for $d=1$, we shall later derive a more
explicit characterisation of eigenmeasures that are periodic or
are pure point measures with uniformly discrete support.

\begin{remark}\label{remark:sub decomp}
   Let us give some further connections.
\begin{enumerate}\itemsep=2pt
\item Proposition~\ref{prop:plambda} {and
    Theorem~\ref{thm:decomp-eigenspaces} hold} for the space of all
  measures that are strictly transformable, as well as for
  $L^2(\RR^d)$, as mentioned earlier.
\item As the authors learned from a correspondence with Feichtinger
  after submitting the first version of this paper, this type of
  result holds for any $\cF\nts$-invariant subspace of $S_0' (\RR^d)$,
  where $S_0' (\RR^d)$ is the dual of the Feichtinger algebra
  $S^{}_0 (\RR^d)$; see \cite{RSteg} and references therein for
  background and further details. \exend
\end{enumerate}
\end{remark}

Let us return to $d=1$. For $r,s \in \RR$ and $\alpha > 0$, we now
define
\begin{equation}\label{eq:def-Z}
   \cZ^{}_{r,s,\alpha} \, \defeq \, \chi^{}_{r}
   \cdot \bigl( \delta^{}_{s} \nts * \delta^{}_{\alpha \ZZ} \bigr) ,
\end{equation}
which is a translation-bounded (hence tempered) complex measure on
$\RR$ that is also transformable, as well as strictly transformable.

\begin{lemma}\label{lem:Z}
  For arbitrary\/ $r,s \in \RR$ and\/ $\alpha > 0$, the measure\/
  $\cZ^{}_{r,s,\alpha}$ from \eqref{eq:def-Z} is tempered and
  $($strictly$\, )$ transformable, and has the following properties.
\begin{enumerate}\itemsep=2pt
\item $\cZ^{}_{r,s + m \alpha,\alpha} = \cZ^{}_{r,s,\alpha}$
    holds for all\/ $m\in\ZZ$.
\item $\cZ^{}_{r \ts + \frac{m}{\alpha},s,\alpha} =
   \ee^{2 \pi \ii \frac{m s}{\alpha}}
   \cZ^{}_{r,s,\alpha}$ holds for all\/ $m\in\ZZ$.
\item $\cF (\cZ^{}_{r,s,\alpha} ) = \delta^{}_{r} \nts * \bigl(
   \chi^{}_{-s} \cdot \frac{1}{\alpha} \, \delta^{}_{\ZZ/\nts\alpha}\bigr)
   = \frac{1}{\alpha} \, \ee^{2 \pi \ii \ts r s} \, \cZ^{}_{-s, r,
     \frac{1}{\alpha}}$.
\item $\cF^2 (\cZ^{}_{r,s,\alpha}) = I \nts \nts . \ts \cZ^{}_{r,s,\alpha}
      =\cZ^{}_{-r,-s,\alpha}$.
\end{enumerate}
Moreover, when\/ $\alpha^2 = \frac{1}{n}$ for some\/ $n\in\NN$, the
decomposition\/
$\cZ^{}_{r,s,\alpha} = \sum_{m=0}^{n-1}
\cZ^{}_{r,s+m\alpha,1/\nts\alpha}$ holds for all\/ $r,s \in \RR$. More
generally, for\/ $\alpha^2 = \frac{p}{q}$ with\/ $p,q\in\NN$ coprime
and\/ $\beta=\sqrt{p q \ts}$, one has
\[
     \cZ^{}_{r,s,\alpha} \, = \sum_{m=0}^{q-1}
    \cZ^{}_{r,s+m\alpha,\beta} \ts .
\]
\end{lemma}

\begin{proof}
  The first relation is obvious, while the second follows from the
  fact that $\ZZ/\nts\alpha$ is dual to $\alpha\ZZ$ as a lattice. It
  is clear that $\cZ^{}_{r,s,\alpha}$ is translation bounded and
  (strictly) transformable.  Property (3) follows from the convolution
  theorem \cite[Thm.~8.5]{TAO1}, applied twice, and an extension of
  Eq.~\eqref{eq:char-1} from $\ZZ$ to $\alpha\ZZ$. Next, observing
  $I\nts \nts .\ts \delta_x = \delta_{-x}$ and
  $\chi^{}_{r} (-x) = \chi^{}_{-r} (x)$, Property (4) follows from a
  simple calculation.

  When $\alpha^2 = \frac{1}{n}$, we have
  $\alpha \ZZ = \ZZ/\nts\nts \sqrt{n}$ and
  $\ZZ/\nts \alpha = \sqrt{n} \ts\ts \ZZ$. As $n \ZZ$ is a sublattice
  of $\ZZ$ of index $n$, the summation formula follows by a standard
  coset decomposition of $\alpha\ts \ZZ$ modulo $\ZZ/\nts \alpha$.

  In the general case, one has $\alpha = \beta/\nts q$, which implies
  $\alpha \ZZ = \dot\bigcup_{m=0}^{q-1} \bigl( \beta \ZZ +
  \frac{m}{q}\beta \bigr)$.  This gives
\[
  \cZ^{}_{r,s,\alpha} \,  = \,
    \chi^{}_{r} \nts \cdot \nts \bigl( \delta^{}_{s}
    * \delta^{}_{\alpha \ZZ} \bigr) \, = \, \chi^{}_{r}
    \nts \cdot \nts \Bigl( \delta^{}_{s} * \! \sum_{m=0}^{q-1}
    \delta^{}_{\beta \ZZ + m \alpha} \Bigr)
    \, = \sum_{m=0}^{q-1} \chi^{}_{r} \nts \cdot \nts
    \bigl( \delta^{}_{s + m \alpha}
    * \delta^{}_{\beta\ZZ} \bigr) \, = \sum_{m=0}^{q-1}
  \cZ^{}_{r, s+m\alpha, \beta}
\]
  for any $r,s \in \RR$  as claimed.
\end{proof}

\begin{remark}
  The measure $\cZ^{}_{r,s,\alpha}$ is closely related to the
  widely-used Zak transform of \cite{Jan}, which is given by
  $(Zf)(\tau,\Omega)=\sum_{k\in\ZZ} f(\tau+k)\, \ee^{-2\pi\ii
    k\Omega}$, for all $\tau,\Omega\in\RR$. A short computation
  reveals
\[
    \cZ^{}_{r,s,1}(f)\, = \, \ee^{2\pi\ii rs}\cdot
    \bigl( Zf \bigr) (s,-r) \ts .
\]
Some of the results of Lemma~\ref{lem:Z} (at least for $\alpha=1$) can
be found in \cite{Jan}, formulated in terms of the Zak transform.

There is a large body of literature on this subject, also in relation
to the Weyl--Heisenberg transform and their discrete versions.
However, to the best of our knowledge, these works neither address
general questions around the transformability of unbounded measures
nor is the transform used to construct eigenmeasures under $\cF$.
\exend
\end{remark}

For $\mu = \delta^{}_{\ZZ}$, the measure $\omega^{}_{s}$ from
Lemma~\ref{lem:modif} is $\cZ^{}_{s,s,1}$, and gives the $4$-cycle
from \eqref{eq:cycle} as
\[
  \cZ^{}_{s,s,1} \, \xrightarrow{\, \cF \,} \, \ee^{2 \pi \ii s^2}
  \cZ^{}_{-s,s,1} \, \xrightarrow{\, \cF \,} \, \cZ^{}_{-s,-s,1} \,
  \xrightarrow{\, \cF \,} \, \ee^{2 \pi \ii s^2} \cZ^{}_{s,-s,1} \,
  \xrightarrow{\, \cF \,} \, \cZ^{}_{s,s,1} \ts .
\]
With this, one can choose an integer $m\in \{ 0, 1, 2, 3 \}$ and
consider the measure
\[
   \nu^{}_{m} \, \defeq \, \ee^{-\pi \ii s^2} \cZ^{}_{s,s,1} +
   \ii^m \ee^{\pi \ii s^2} \cZ^{}_{-s,s,1} +
   \ii^{2m} \ee^{-\pi \ii s^2} \cZ^{}_{-s,-s,1}+
   \ii^{3m} \ee^{\pi \ii s^2} \cZ^{}_{s,-s,1} \ts ,
\]
which satisfies {$\widehat{\nu^{}_{m}} = (-\ii)^m \ts \nu^{}_{m}$}.
When $s\in\QQ$, so $s=\frac{p}{q}$ with $p\in\ZZ$ and $q\in\NN$, the
example is (at least) $q$-periodic, while $s$ irrational provides
aperiodic examples.

\begin{prop}\label{prop:FD}
  Let\/ $\alpha>1$ be fixed, and let\/ $K$ and\/ $J$ be bounded
  fundamental domains for the lattices\/ $\alpha \ZZ$ and\/
  $\frac{1}{\alpha} \ZZ$, respectively. Then,
  $B_{\alpha} \defeq \{ \cZ^{}_{r,s,\alpha} : r\in J, s \in K \}$ is
  an algebraic basis for\/
  $V_{\alpha} \defeq \mathrm{span}^{}_{\CC}\ts \{ \cZ^{}_{r,s,\alpha}
  : r,s \in \RR \}$.
\end{prop}

\begin{proof}
  By Lemma~\ref{lem:Z}{\ts}(1) and (2), it is clear that the elements
  of $B_{\alpha}$ are distinct and that $B_{\alpha}$ is a spanning set
  for $V_{\alpha}$, which is the complex vector space of all
  \emph{finite} linear combinations of measures $\cZ^{}_{r,s,\alpha}$
  with fixed $\alpha$.  It thus remains to show linear independence of
  the elements of $B_{\alpha}$. So, let $n\in\NN$ and assume
\begin{equation}\label{eq:lin}
  c^{}_{1} \cZ^{}_{r^{}_1, s^{}_1, \alpha} + \dots +
  c^{}_{n} \cZ^{}_{r^{}_n, s^{}_n, \alpha} \, = \, 0
\end{equation}
with $s_i \in K$ and $r_j \in J$ such that the $n$ pairs
$(s_i , r_i )$ are distinct. Since $K$ is a proper funda\-mental
domain for $\alpha \ZZ$, for any $1 \leqslant k,\ell \leqslant n$, we
have either $s^{}_{k} = s^{}_{\ell}$ or
$(s^{}_{k} + \alpha \ZZ) \cap (s^{}_{\ell} + \alpha \ZZ) =
\varnothing$.

Select some $1\leqslant k \leqslant n$ and set
$F^{}_{k} = \{ 1 \leqslant \ell \leqslant n : s^{}_{\ell} = s^{}_{k}
\}$. Restricting \eqref{eq:lin} to $s^{}_{k} + \alpha \ZZ$ gives
\[
  \sum_{\ell \in F^{}_{k}}  c^{}_{\ell} \ts
  \cZ^{}_{r^{}_{\nts\ell}, s^{}_{k}, \alpha} \, = \, 0 \ts .
\]
Taking the Fourier transform turns this relation into
\begin{equation}\label{eq:lin-3}
  \sum_{\ell\in F^{}_{k}} \tfrac{c^{}_{\ell}}{\alpha}
  \,\ee^{2 \pi \ii r^{}_{\nts \ell} s^{}_{k}}
  \cZ^{}_{- s^{}_{k}, r^{}_{\nts \ell}, \frac{1}{\alpha}}
  \, = \, 0 \ts .
\end{equation}
{Next, we note that, for all $\ell \in F^{}_{k}$ with $\ell \ne k$, we
  have $r^{}_{\nts \ell} \ne r^{}_{k}$, as $s^{}_{\nts \ell} = s^{}_{k}$ and
  $(s^{}_{\nts \ell}, r^{}_{\nts \ell}) \ne (s^{}_{k}, r^{}_{k})$.} Consequently,
as $r^{}_{k}, r^{}_{\nts \ell} \in J$, we have
\[
  (r^{}_{k} + \tfrac{1}{\alpha} \ZZ) \cap
  (r^{}_{\nts\ell} + \tfrac{1}{\alpha} \ZZ)
  \, = \, \varnothing \ts .
\]
Now, restricting \eqref{eq:lin-3} to $r^{}_{k} + \frac{1}{\alpha} \ZZ$
gives
\[
  \tfrac{c^{}_k}{\alpha} \ee^{2 \pi \ii r^{}_{k} s^{}_{k} }
  \cZ^{}_{-s^{}_{k}, r^{}_{k}, \frac{1}{\alpha}} \, = \, 0 \ts ,
\]
which implies $c^{}_{k} = 0$. Since $k$ was arbitrary, we are done.
\end{proof}

Let us now look into classes of periodic eigenmeasures more
systematically.

\section{Lattice-periodic eigenmeasures}\label{sec:cryst}

If $\vG$ is a lattice in $\RR^d$, and $\delta^{}_{\nts\vG}$ the
corresponding Dirac comb, the PSF from \eqref{eq:PSF-1} becomes
\begin{equation}\label{eq:PSF-2}
    \widehat{\delta^{}_{\nts\vG}} \, = \, \dens (\vG)
    \, \delta^{}_{\vG^*} \ts ,
\end{equation}
where $\dens (\vG)$ is the (uniformly existing) density of $\vG$, and
$\vG^*$ denotes the dual lattice, as given by
$\vG^* = \{ x \in \RR^d : x y \in \ZZ \text{ for all } y \in \vG\ts\}$;
compare \cite[Thm.~9.1]{TAO1}.

\begin{fact}\label{fact:lattice}
  If\/ $\vG\subset \RR^d$ is a lattice, the Dirac comb\/
  $\delta^{}_{\nts\vG}$ is an eigenmeasure for\/ $\cF$ if and only
  if\/ $\vG$ is self-dual, that is, $\vG^* = \vG$, which also
  implies\/ $\dens (\vG) = 1$.
\end{fact}

\begin{proof}
  $ \widehat{\delta^{}_{\nts\vG}} = \lambda\ts \delta^{}_{\nts\vG}$ means
  $\vG^*=\vG$ and $\lambda = \dens (\vG)$, via the PSF, together with
  $\lambda^4=1$. As $\lambda=1$ is the only positive, real solution,
  we get $\dens (\vL) = 1$ and $\vG^* = \vG$ as claimed.
\end{proof}

To go beyond this case, let us start with $d=1$ and assume $\mu$ to be
a measure on $\RR$ that is $\alpha$-periodic, for some $\alpha>0$,
which is to say that $\delta^{}_{\alpha} \nts * \mu = \mu$. Such a
measure is of the form
\begin{equation}\label{eq:cryst-1}
     \mu \, = \, \varrho * \delta^{}_{\alpha \ZZ} \ts ,
\end{equation}
where $\varrho$ is a \emph{finite} measure; compare
\cite[Sec.~9.2.3]{TAO1}. Without loss of generality, we may assume
that $\supp (\varrho) \subseteq [0,\alpha]$ with
$\varrho( \{ \alpha \}) = 0$, for instance by setting
$\varrho = \mu |^{}_{[0,\alpha)}$. As the periodic repetition of a
finite motif, $\mu$ obviously is translation bounded, hence also
tempered. Any such measure is (strictly) transformable
\cite[Cor.~6.1]{ARMA1}.

\begin{lemma}\label{lem:cryst}
  Let\/ $\alpha>0$ be fixed, and let\/ $\mu\ne 0$ be a tempered
  measure on\/ $\RR$ that is\/ $\alpha$-periodic.  If\/
  $\widehat{\mu} = \lambda \mu$, one has\/
  $\lambda \in \{ 1, \ii, -1, -\ii \}$ together with\/ $\alpha^2 = n$
  for some\/ $n\in\NN$. In particular, any such measure must be of the
  form \eqref{eq:cryst-1}, where the finite measure\/ $\varrho$ can be
  chosen to have support in\/ $[0,\alpha) \cap \ZZ/\nts\alpha $, that
  is, in the natural coset representatives of\/ $\ZZ/\nts\alpha$
  modulo\/ $\alpha \ZZ$.
\end{lemma}

\begin{proof}
  Under the assumption, we may use the representation from
  \eqref{eq:cryst-1}, together with the choice of $\varrho$.  From
  here, we see that $\mu$ is strictly transformable, and the
  convolution theorem \cite[Thm.~8.5]{TAO1} in conjunction with the
  general PSF from Eq.~\eqref{eq:PSF-2} gives
\begin{equation}\label{eq:cryst-2}
    \widehat{\mu} \, = \, \alpha^{-1}
    \widehat{\varrho} \,
    \delta^{}_{\ZZ/\nts\alpha} \ts .
\end{equation}
Since $\widehat{\varrho}\ts$ is a continuous function on $\RR$, this
entails that $\supp (\widehat{\mu}) \subseteq \ZZ/\nts\alpha$, where
$\ZZ/\alpha$ is a lattice and thus uniformly discrete as a point set.

If $\widehat{\mu} =\lambda \mu$, where $\mu$ is non-trivial, we know
that $\lambda^4=1$. Now, comparing the representation of $\mu$ from
\eqref{eq:cryst-1} with Eq.~\eqref{eq:cryst-2} implies
$\supp(\mu) = \supp(\varrho) + \alpha \ZZ \subseteq \ZZ/\nts\alpha$,
hence
\[
    \alpha \ts \supp(\varrho) + \alpha^2 \ZZ
    \, \subseteq \, \ZZ \ts ,
\]
where $A+B$ denotes the Minkowski sum of two sets.  This inclusion
implies $\alpha \supp(\varrho) \subseteq \ZZ$ because
$0\in \alpha^2 \ZZ$.  Since $\alpha \supp (\varrho) \ne \varnothing$,
it contains some integer, $m$ say, which now results in
$\alpha^2 \ZZ \subseteq \ZZ$.  But this means that $\alpha^2 = n$ must
be an integer, and $\supp(\varrho) \subset \ZZ/\nts\alpha$.

Under our assumptions, and with our choice of $\varrho$, we get
$\supp (\varrho) \subseteq [0,\alpha) \cap \ZZ/\nts\alpha$, so
$\varrho$ must be a pure point measure with support in the finite set
$\{ \frac{m}{\alpha} : 0 \leqslant m < n \}$, which establishes the
last claim.
\end{proof}

This provides a class of Dirac combs with pure point diffraction
\cite{BM}, with the additional property that they are doubly sparse in
the sense of \cite{LO1,LO2,BST},  see also \cite{Kol}.

\subsection{Connection with discrete Fourier transform}

Let $\alpha = \sqrt{n}$ with $n\in \NN$ be fixed, and consider the
pure point measure
\begin{equation}\label{eq:mu-DFT}
    \mu \, = \, \varrho * \delta^{}_{\alpha \ZZ}
    \quad \text{with} \quad  \varrho \, =
    \sum_{m=0}^{n-1} c^{}_{m} \, \delta^{}_{m/\nts\alpha} \ts .
\end{equation}
Clearly, $\mu$ is \emph{crystallographic} in the sense of \cite{TAO1}.
Indeed, it is $\alpha$-periodic, but might also have smaller periods,
depending on the coefficients $c^{}_{m}$.  Moreover, $\mu$ is
transformable, with
\[
  \widehat{\mu} \, = \, \widehat{\varrho} \:
  \widehat{\delta^{}_{\alpha \ZZ}} \, = \, \alpha^{-1}
  \sum_{m=0}^{n-1} c^{}_{m} \, \ee^{- 2 \pi \ii \frac{m}{\alpha} (.)}
  \, \delta^{}_{\ZZ/\nts\alpha} \ts .
\]
Observing that
$\delta^{}_{\ZZ/\nts\alpha} = \delta^{}_{\alpha \ZZ} *
\sum_{\ell=0}^{n-1} \delta^{}_{\ell/\nts\alpha}$, one sees that the
value of the character in the above sum only depends on the coset of
$\alpha \ZZ$, thus giving the simplification
\begin{equation}\label{eq:mu-DFT-2}
   \widehat{\mu} \, = \, \delta^{}_{\alpha \ZZ} \ts  *
   \sum_{\ell=0}^{n-1} \biggl(\! \myfrac{1}{\sqrt{n}}
   \sum_{m=0}^{n-1} \ee^{- 2 \pi \ii \frac{m \ell }{n}} \, c^{}_{m}
   \!\biggr) \delta^{}_{\ell/\nts\alpha}  \ts .
\end{equation}
The crucial observation to make here is that the term in brackets is
the \emph{discrete Fourier transform} (DFT) of the vector
$c = (c^{}_{0}, \ldots , c^{}_{n-1} )^{T}$.

If $U_n$ denotes the unitary DFT or Fourier matrix, it is well known
that its eigenvalues satisfy $\lambda \in \{ 1, \ii, -1, -\ii
\}$. Their multiplicities, in closed terms, were already known to
Gauss, while the eigenvectors are difficult,\footnote{A good summary
  with all the relevant details and references can be found in the
  \textsc{WikipediA} entry on the DFT, which is recommended for
  background. Some results are collected in our Appendix.}  and still
not known in closed form; see \cite{Gruenbaum} for a version that
relates to the Hermite functions of the continuous case.

\begin{theorem}\label{thm:lattice}
  Let\/ $\alpha > 0$ be fixed, and let\/ $\lambda$ be a fourth root of
  unity. Further, let\/ $\mu\ne 0$ be a measure on\/ $\RR$ that is\/
  $\alpha$-periodic. Then, the following properties are equivalent.
\begin{enumerate}\itemsep=2pt
\item $\mu$ is an eigenmeasure of\/ $\cF$ with eigenvalue\/ $\lambda$.
\item $\mu$ is a strict eigenmeasure of\/ $\cF$ with eigenvalue\/
  $\lambda$.
\item The measure\/ $\mu$ has the form \eqref{eq:mu-DFT} with\/
  $\alpha = \sqrt{n}$, for some\/ $n\in\NN$ and some eigenvector\/
  $c= (c^{}_{0}, \ldots , c^{}_{n-1} )^{T} \in \CC^n$ of the
  corresponding DFT, meaning\/ $U_n c = \lambda c$.
\end{enumerate}
\end{theorem}

\begin{proof}
  As pointed out above, any $\alpha$-periodic measure is translation
  bounded. Then, $(1) \Leftrightarrow (2)$ is a consequence of
  Fact~\ref{fact:transformable}.

  The equivalence $(1) \Leftrightarrow (3)$ follows from
  Lemma~\ref{lem:cryst} by a comparison of coefficients in
  \eqref{eq:mu-DFT} and \eqref{eq:mu-DFT-2}.
\end{proof}

\begin{example}\label{ex:DFT-1}
  When $\alpha^2=1$ in Theorem~\ref{thm:lattice}, the only solution we
  get (up to a constant) is $\mu=\delta^{}_{\ZZ}$, with $\lambda = 1$,
  as expected.

  When $\alpha^2=2$, so $\alpha = \sqrt{2}$, the measure $\varrho$ in
  \eqref{eq:cryst-1} has the form
  $\varrho = c^{}_{0} \ts \delta^{}_{0} + c^{}_{1} \ts
  \delta^{}_{1/\nts\alpha}$.  Here, with $c^{}_{0} = 1 \pm \sqrt{2}$
  and $c^{}_{1} = 1$, one gets eigenmeasures with $\lambda = \pm 1$,
  which are the only possibilities (up to an overall
  constant). Examples with larger values of $\alpha$ can be
  constructed with the material given in the Appendix.  \exend
\end{example}

\begin{remark}\label{rem:symm}
  Note that Fact~\ref{fact:symm} has consequences on the structure of
  the eigenvectors $c$ of the DFT. Indeed, it is clear that we get
  $c^{}_{0} = \lambda^2 c^{}_{0}$ together with
 \[
     c^{}_{n-k} \, = \lambda^2 c^{}_{k} \, ,
     \quad \text{for } 1 \leqslant k \leqslant n-1 \ts .
 \]
 In particular, for $\lambda = \pm 1$, the (partial) vector
 $(c^{}_{1}, \ldots , c^{}_{n-1})$ is palindromic, while
 $\lambda = \pm \ts \ii$ implies $c^{}_{0} = 0$ together with
 $(c^{}_{1}, \ldots , c^{}_{n-1})$ being skew-palindromic.  This is
 clearly visible in the above examples and in those of the Appendix.
 \exend
\end{remark}

Let us next consider another class of eigenmeasures that implicitly
emerge from a \emph{cut and project scheme} (CPS) in the theory of
aperiodic order; see \cite{TAO1} for background. This will provide
non-periodic examples of pure point eigenmeasures.

\section{Shadows of product measures}\label{sec:shadows}

Let us return to the simple statement from Fact~\ref{fact:lattice},
and extend it into a different direction.  Here and below, we always
assume $\RR^d$ to be equipped with Lebesgue measure as its (standard)
Haar measure, that is, the unique translation-invariant measure which
gives volume $1$ to the unit cube. Clearly, $\vG$ can only be
self-dual if $\dens (\vG) = 1$, which is to say that its fundamental
domain has unit volume. Note that this is not a sufficient criterion
for $d>1$.

More generally, when $B$ denotes a basis matrix for $\vG$, with the
columns being the basis vectors in Cartesian coordinates,
$B^* \defeq (B^{-1})^{T}$ is the \emph{dual matrix}, and the fitting
basis matrix for $\vG^*$. In this formulation, self-duality of $\vG$
is equivalent with orthogonality of the matrix $B$. It is now easy to
check that the only self-dual lattices in $\RR^2$ are the square
lattices, that is, $\ZZ^2$ and its rotated siblings. A natural choice
for the basis matrix thus is
\begin{equation}\label{eq:def-B}
  B \, = \, \begin{pmatrix}
     \cos (\theta ) & -\sin (\theta) \\
     \sin (\theta) & \cos (\theta) \end{pmatrix}
     \, = \, \bigl( u^{}_{\theta}, v^{}_{\theta} \bigr) ,
\end{equation}
where $u^{}_{\theta}$ and $v^{}_{\theta}$ are the column vectors. We
denote the corresponding lattice by $\vG^{}_{\theta}$. Note that the
parameter can be restricted to $0 \leqslant \theta < \frac{\pi}{2}$,
which covers all cases once, due to the fourfold rotational symmetry
of the square lattice. Given this choice of basis, we can always
uniquely write $z\in\vG^{}_{\theta}$ as
$z = m \ts u^{}_{\theta} + n \ts v^{}_{\theta}$ with $m,n \in \ZZ$.
We thus map $z\in\vG^{}_{\theta}$ to a pair
$(m^{}_{z},n^{}_{z}) \in \ZZ^2$ via this choice of basis.

The next result is standard \cite[Thm.~VII.XIV]{Schw}, and follows
from a simple Fubini-type calculation with the product Lebesgue
measure on $\RR^{p+q} = \RR^p \!  \times \nts \RR^q$, where the
function
$f\nts\otimes g \nts : \, \RR^p \! \times \nts \RR^q
\xrightarrow{\quad} \CC$ is defined by
$f\nts\otimes g \, (x,y) \defeq f(x) \, g(y)$ as usual.

\begin{fact}\label{fact:tensor}
  Let\/ $p,q\in\NN$, $f\in L^{1} (\RR^p)$, $g\in L^{1} (\RR^q)$, and
  consider the function\/ $h \defeq f\nts\otimes g$. Then,
  $h \in L^{1} (\RR^{p+q})$, and its Fourier transform reads\/
  $\widehat{h} = \widehat{f} \nts\otimes \widehat{g}$. In particular,
  this applies to\/ $f\in\cS (\RR^{p})$ and\/ $g\in\cS (\RR^{q})$,
  with\/ $h \in \cS (\RR^{p+q})$. \qed
\end{fact}

Now, let $\vG \subset \RR^{d+m}$ be a lattice, let $g \in \cS(\RR^m)$
be fixed, and consider
\begin{equation}\label{eq:def-kamm}
  \delta^{\ts\star}_{\nts g,\vG} (f) \, \defeq \,
  \delta^{}_{\nts\vG} (f\nts \otimes g) \ts .
\end{equation}
Since the mapping
$\cS(\RR^d) \ni f \mapsto f \otimes g \in \cS(\RR^{d+m})$ is
continuous with respect to the topology of Schwartz space, and since
$\delta^{}_{\nts\vG}$ is a tempered distribution, it is clear that
$\delta^{\star}_{g,\vG}$ is a tempered distribution as well. In fact,
we have more as follows.

\begin{prop}\label{prop:A}
  The measure\/ $\delta^{\ts\star_{}}_{\nts g,\vG}$ from
  \eqref{eq:def-kamm} is translation bounded and transformable, with
  the transform\/
  $\widehat{\delta^{\ts\star_{}}_{\nts g,\vG}} = \dens(\vG)\,
  \delta^{\ts\star_{}}_{\widecheck{g},\vG^{*}_{\vphantom{t}}} $.
\end{prop}

\begin{proof}
  If $\vG\subset\RR^{d+m}$ is a lattice and $h \in \cS(\RR^{m+d})$,
  one has $ \sum_{z \in \vG} \lvert h(z)\rvert < \infty $ by standard
  arguments \cite{Schw}, hence
\begin{equation}\label{eq:something}
   \sum_{(x,y) \in \vG} \lvert f(x)\, g(y) \rvert \, < \, \infty
\end{equation}
holds for all $f \in \cS(\RR^d)$ and $g \in \cS(\RR^m)$.
Consequently, also for all $\varphi \in \Cc(\RR^d)$, we have
\[
  \sum_{(x,y) \in \vG} \lvert \varphi(x) \, g(y) \rvert
  \, < \, \infty \ts .
\]
We can thus define $\mu \colon \Cc(\RR^d) \longrightarrow \CC$ by
\[
  \mu(\varphi) \: =  \sum_{(x,y) \in \vG} g(y)
  \,  \delta_x (\varphi) \: = \sum_{(x,y) \in \vG}
  \varphi(x) \, g(y) \ts ,
\]
with the sum being absolutely convergent. Moreover, by
Eq.~\eqref{eq:something}, $\mu(f)$ is well defined for all
$f \in \cS (\RR^d)$ and satisfies
$ \mu(f) = \delta^{\ts\star}_{\nts g,\vG} (f)$.

Now, we show that $\mu$ is also a measure, which implies that
$\delta^{\ts\star_{}}_{\nts g,\vG} $ is a measure as well. It is
obvious that $\mu$ is a linear mapping.  Next, let $K \subset \RR^d$
be a fixed compact set, select $f \in \Cc^\infty(\RR^d)$ so that
$f \geqslant 1^{}_K$, and consider
\[
  C^{}_K \, \defeq \sum_{(x,y) \in \vG}
  \lvert f(x) \, g(y) \rvert \, < \, \infty \ts .
\]
Then, if $\varphi \in \Cc(\RR^d)$ satisfies
$\supp(\varphi)\subseteq K$, we get
\[
  \lvert \mu(\varphi) \vert \, = \, \biggl|
  \sum_{(x,y) \in \vG}  \varphi(x) \, g(y) \biggr| \, \leqslant
  \sum_{(x,y) \in \vG}  \lvert \varphi(x) \, g(y) \vert \,
  \leqslant  \, \| \varphi\|^{}_{\infty}\! \sum_{(x,y) \in \vG}
   \lvert f(x) \, g(y) \rvert
  \, \leqslant \, C^{}_K \, \| \varphi \|^{}_{\infty}
\]
which shows that $\mu$ is a measure.

Next, by applying the proof of \cite[Cor.~2.1]{ST} to
$h=f \nts\otimes g$, we see that there is a constant $C$ such that
$\bigl( \lvert \delta_{\nts\vG} \rvert * \lvert I\nts .h \rvert \bigr) (x)
\leqslant C$ holds for all $x \in \RR^{d+m}$.  In particular, for all
$t \in \RR^d$, we have
\[
\begin{split}
  \lvert \mu \rvert (t+K) \, & = \sum_{\substack{(x,y) \in \vG \\ x
      \in t+K}} \lvert g(y) \rvert
  \, \leqslant  \sum_{(x,y) \in \vG } \vert f(x-t) \, g(y) \rvert \\[2mm]
  & = \sum_{(x,y) \in \vG} \vert h (x-t,y) \rvert \, = \, \bigl(\lvert
  \delta^{}_{\nts \vG} \rvert * \lvert I \nts .h \rvert \bigr) (t,0)
  \, \leqslant \, C ,
\end{split}
\]
which shows that $\mu$ is translation bounded.

The verification of the last claim is similar to an argument from
\cite{RS}. Indeed, the PSF implies that, for all $f\in \cS(\RR^d)$, we
have
\[
  \widehat{\delta^{\ts\star_{}}_{\nts g,\vG}}(f) \, = \,
  \delta^{\ts\star}_{g,\vG} \bigl( \widehat{f}\, \bigr) \, = \,
  \delta^{}_{\nts\vG} \bigl( \widehat{f} \nts\otimes g \bigr) \, = \,
  \dens(\vG) \, \delta^{}_{\nts\vG^{*}_{}} \bigl( f \nts \otimes
  \widecheck{g} \ts \bigr) \, = \, \dens(\vG)\,
  \delta^{\ts\star}_{\widecheck{g},\vG^{*}_{}}(f)
\]
with the last distribution being a measure by the first part of the
proof.
\end{proof}

Now, we can proceed as follows. Let $\lambda$ be any fixed fourth root
of unity, and select a Schwartz function $g \in \cS (\RR)$ such that
$\widehat{g} = \lambda \ts g$, which we know to exist via the Hermite
functions. Clearly, this implies
$\widecheck{g} = \overline{\lambda} \ts g$, because we have
$\widehat{\nts\widehat{g}\ts} = g \circ I$ together with
$\lambda^4=1$.

Next, fix a parameter $0\leqslant \theta < \frac{\pi}{2}$ and consider
the lattice $\vG^{}_{\theta}$ in $\RR^2 = \RR \times \RR $. Let
$f\nts\otimes g$ refer to the product function with respect to this
(Cartesian) splitting, so
\[
  f \nts\otimes g \, (z) \, = \, f \nts\otimes g \,
  (m^{}_{z} u^{}_{\theta} + n^{}_{z} v^{}_{\theta}) \, = \,
  f \bigl(m^{}_{z} \cos (\theta) - n^{}_{z} \sin (\theta)\bigr)
  \cdot g \bigl(m^{}_{z} \sin (\theta) + n^{}_{z} \cos (\theta)\bigr),
\]
which elucidates the role of the angle $\theta$.  Now, for a fixed
lattice $\vG^{}_{\theta}$, we consider the translation-bounded (and
hence tempered) measure $\omega^{}_{g}$ on $\RR$ defined by
\begin{equation}\label{eq:def-om}
  \omega^{}_{g} ( \varphi) \, \defeq
  \sum_{z\in \vG^{}_{\theta}} \varphi\otimes\nts g \, (z)
  \, = \, \delta^{\ts\star_{}}_{\nts g,\vG^{}_{\theta}} (\varphi) \ts ,
\end{equation}
where $\varphi$ is an arbitrary Schwartz function. In fact,
\eqref{eq:def-om}
is well-defined for $\varphi\in C_{\mathsf{c}} (\RR)$, too.
Then, the (distributional) transform {gives}
\[
\begin{split}
   \widehat{\omega^{}_{g}} (\varphi) \, & = \, \omega^{}_{g}
   (\ts\widehat{\varphi}\ts ) \, = \sum_{z \in \vG^{}_{\theta}} \!
   \widehat{\varphi} \otimes\nts g \, (z) \, =
   \sum_{z \in \vG^{}_{\theta}}
   \widehat{\nts\varphi \otimes \nts \widecheck{g}\ts} \, (z)\\[2mm]
    & = \, \sum_{z \in \vG^{}_{\theta}}  \varphi \otimes
   \nts \widecheck{g} \, (z)
   \, = \, \overline{\lambda} \sum_{z \in \vG^{}_{\theta}}
   \varphi \otimes \nts g \, (z) \, = \, \overline{\lambda} \,
   \omega^{}_{g} (\varphi) \ts .
\end{split}
\]
Here, we have used Fact~\ref{fact:tensor} for $p=q=1$ in the upper
line, while the first equality in the second line follows from the PSF
for the self-dual lattice $\vG^{}_{\theta}$.  Since $\varphi$ is
arbitrary, this implies the relation
$\widehat{\omega^{}_{g}} = \overline{\lambda} \,
\omega^{}_{g}$. Invoking \cite[Lemma~5.2 and Thm.~5.2]{Nicu2}, we see
that the measure $\omega^{}_{g}$ from \eqref{eq:def-om} is translation
bounded and that
$\widehat{\omega^{}_{g}} = \overline{\lambda} \, \omega^{}_{g}$ also
holds in the strict sense.  Consequently, we can construct
eigenmeasures of $\cF$ for any fourth root of unity this way. Let us
summarise this derivation as follows.

\begin{theorem}\label{thm:aper-meas}
  Let\/ $\vG^{}_{\theta}$ be the self-dual, planar lattice defined by
  the basis matrix\/ $B$ from \eqref{eq:def-B}, and let\/ $g \ne 0$ be
  a Schwartz function with\/ $\widehat{g} = \lambda \ts g$ for some\/
  $\lambda \in \{1, \ii, -1, -\ii \}$. Then, the tempered measure\/
  $\omega^{}_{g}$ defined by \eqref{eq:def-om} is a\/ $($strict$\, )$
  eigenmeasure of\/ $\cF\!$, with\/
  $\widehat{\omega^{}_{g}} = \overline{\lambda} \, \omega^{}_{g}$.
  \qed
\end{theorem}

\begin{remark}
  As is clear from \cite[Cor.~VII.2.6]{SW}, this approach also works
  when $g$ is continuous and decays such that
  $\varphi \otimes \nts g \, (z) = \cO \bigl( (1 + \lvert z
  \rvert)^{-2-\epsilon} \bigr)$ as $\lvert z \rvert \to\infty$, for some
  $\epsilon > 0$.
\exend
\end{remark}

Depending on the parameter $\theta$, the lattice $\vG^{}_{\theta}$ may
be in rational or irrational position relative to the horizontal
line. The rational case means $\tan (\theta ) = \frac{p}{q}$ with
integers $0 \leqslant p \leqslant q$ that can be chosen coprime, with
$q=1$ when $p=0$. Rationality implies
$\sin (\theta )^2 = p^2/ (p^2 + q^2)$ and
$\cos (\theta )^2 = q^2 / (p^2 + q^2)$.  In any such case, the
intersection of $\vG^{}_{\theta}$ with the horizontal line is a
one-dimensional lattice. Observing that
$q \sin (\theta ) - p \cos (\theta ) = 0$, this lattice is
$\alpha \ZZ$, with
\[
  \alpha \, = \, \frac{q \cos (\theta ) + p \sin (\theta )}{q}
  \, = \, \frac{\sqrt{p^2 + q^2}}{q}
\]
when $q \neq 0$, and $\alpha=1$ when $q=0$.  Consequently,
$\omega^{}_{g}$ is $\alpha$-periodic in this case, and we are back to
the class discussed in Section~\ref{sec:cryst}.

\begin{example}\label{ex:0}
  When $\theta = 0$, we get $\vG^{}_{0} = \ZZ^2$, and our Radon
  measure simplifies to $\omega^{}_{g} = c \, \delta^{}_{\ZZ}$, with
  $c = \sum_{m\in\ZZ} g(m)$. Since this measure is $1$-periodic, it is
  either an eigenmeasure for eigenvalue $1$, or it must be trivial. If
  $g$ was chosen as a Hermite function, $g = h_n$ say, we thus get an
  extra condition for any $n \not\equiv 0 \bmod{4}$, namely
\begin{equation}\label{eq:sum-rule}
    \sum_{k\in\ZZ} h_n ( k) \, = \, 0 \ts .
\end{equation}
This is clear by the anti-symmetry of $h_n$ for all odd $n$, while it
looks a little less obvious for $n \equiv 2 \bmod 4$.  However, this
sum rule follows directly from the classic PSF for functions, via
\[
    \sum_{k\in\ZZ} h_n (k) \, = \sum_{k\in\ZZ} \widehat{h_n} (k)
    \, = \, (-\ii)^n \sum_{k\in\ZZ} h_n (k) \ts ,
\]
which implies \eqref{eq:sum-rule} for any $n\not\equiv 0 \bmod{4}$.
\exend
\end{example}

\begin{example}\label{ex:next}
  Consider $p=q=1$, so $\theta = \frac{\pi}{4}$. Here, the
  intersection of the lattice $\vG^{}_{\pi/4}$ with the horizontal
  line is $\sqrt{2} \ts \ZZ$ and the measure becomes
  $\omega^{}_{g} = c^{}_{0} \ts \delta^{}_{\nts\sqrt{2} \ts \ZZ} +
  c^{}_{1} \ts \delta^{}_{\nts\beta + \sqrt{2} \ts \ZZ}$, with
  $\beta = 1/\sqrt{2}$ and coefficients
\[
  c^{}_{0} \, = \sum_{m\in\ZZ} g \bigl( m \sqrt{2} \, \bigr)
  \quad \text{and} \quad
  c^{}_{1} \, = \sum_{m\in\ZZ} g \bigl( m \sqrt{2} + \beta \bigr) .
\]
Since we have $\alpha\ts$-periodicity with $\alpha^2 = 2$, there are
eigenmeasures for $\lambda = \pm\ts 1$. When $\widehat{g} = \pm g$,
this is only consistent if the coefficients satisfy
$c^{}_{0} / c^{}_{1} = 1 \pm \sqrt{2}$, in line with
Example~\ref{ex:DFT-1}. This covers the cases of $g = h_n$ with $n$
even. For odd $n$, the measure $\omega^{}_{g}$ must be trivial, hence
$c^{}_{0} = c^{}_{1} = 0$, which is clear from the anti-symmetry of
the Hermite functions in this case.

More generally, extending this observation to other rational values of
$\theta$, one obtains further conditions from the structure of the
eigenvectors of the DFT.  \exend
\end{example}

When $\tan(\theta)$ is \emph{irrational}, we obtain aperiodic
extensions, which can be understood by viewing the setting as a CPS of
the form $(\RR, \RR, \vG^{}_{\theta})$; see \cite{TAO1} for more on
this concept, and \cite{Mey72,Moo97,Moo00} for general background.
The CPS here can be summarised in the diagram
\begin{equation}\label{eq:CPS}
\renewcommand{\arraystretch}{1.2}\begin{array}{r@{}ccccc@{}l}
   \\  & \RR & \xleftarrow{\;\;\; \pi \;\;\; }
   & \RR \nts\nts \times \nts\nts \RR &
   \xrightarrow{\;\: \pi^{}_{\text{int}} \;\: } & \RR & \\
   & \cup & & \cup & & \cup & \hspace*{-1.5ex}
   \raisebox{1pt}{\text{\footnotesize dense}} \\
   & \pi (\cL) & \xleftarrow{\;\ts 1-1 \;\ts } &
     \cL = \vG^{}_{\theta} &
     \xrightarrow{ \qquad } &\pi^{}_{\text{int}} (\cL) & \\
   & \| & & & & \| & \\
   & L & \multicolumn{3}{c}{\xrightarrow{\qquad\quad\quad
    \,\,\,\star\!\! \qquad\quad\qquad}}
   &  {L_{}}^{\star\nts}  & \\ \\
\end{array}\renewcommand{\arraystretch}{1}
\end{equation}
where $\star$ is the star map of the CPS. Clearly, this diagram is not
restricted to the case $\cL=\vG^{}_{\theta}$, and $\cL$ can be any
lattice in $\RR^2$ subject to the conditions encoded in the CPS.
\smallskip

Combining Theorem~\ref{thm:aper-meas} and
Fact~\ref{fact:transformable}, we can reformulate our previous result
as follows.

\begin{coro}\label{coro:aper}
  Let\/ $(\RR, \RR, \cL)$ be a CPS and let\/ $g \in \cS(\RR)$ be
  arbitrary, but fixed. If the lattice\/ $\cL$ is self-dual and if\/
  $\widecheck{g} = \lambda g$, then necessarily with\/
  $\lambda\in\{ 1, \ii, -1, -\ii \}$, the tempered measure\/
  $\omega^{}_g=\delta^{\,\star}_{\nts g,\cL}$ is an eigenmeasure of\/
  $\cF$ with eigenvalue\/ $\lambda$, also in the strict sense.  \qed
\end{coro}

\smallskip

This result has a nice consequence as follows, still within the CPS
\eqref{eq:CPS}. If $g\in L^{2} (\RR)$ is continuous, the measure
$\omega^{}_{g}$ is well defined. Now, expanding $g$ as a series in
Hermite functions (in the sense explained in
Section~\ref{sec:prelim}), $\omega^{}_{g}$ can be written as a series
in eigenmeasures.  This has the potential to give a new tool for
studying the diffraction of weighted Dirac combs with Meyer set
support, which could potentially work also beyond regular model sets
or weak model sets of maximal density.

\begin{remark}
  Studying the proofs of Lemmas 5.2 and 5.3 in \cite{Nicu2} with care,
  one can see that various of our previous arguments remain true for
  $g \in \Cz (\RR)$, as long as $g$ asymptotically satisfies
  $g (x) = \cO \bigl( (1+\lvert x \rvert )^{2+\epsilon} \bigr)$ for
  some $\epsilon > 0$ as $\lvert x \rvert \to \infty$, while
  Proposition~\ref{prop:A} still holds if $g \in \Cz(\RR)$ has the
  property that both $g$ and $\widecheck{g}$ satisfy such an
  asymptotic behaviour.  \exend
\end{remark}

\section{Eigenmeasures with uniformly discrete
support}\label{sec:discrete}

Let us {next} characterise eigenmeasures on $\RR$ with uniformly
discrete support.  Note that each tempered measure with uniformly
discrete support is strongly tempered \cite[Thm.~4.4]{BS}, which
allows us to use the results of \cite{LO1}. We recall
Eq.~\eqref{eq:def-Z} and begin with another consequence of the cycle
structure induced by $\cF^4 = \id$.

\begin{lemma}\label{lem:cycle}
  Let\/ $r,s \in \RR$ and\/ $n\in\NN$ be fixed, and let\/ $\lambda$
  be a fourth root of unity. Then,
\[
  {\cY^{}_{r,s,\sqrt{n},\lambda} \, \defeq \,
  P^{}_{\lambda} (\cZ^{}_{r,s,\sqrt{n}} )  \, = \,
   \frac{1}{4} \bigl( \id \ts +  \lambda^{-1} \cF +
    \lambda^{-2} \cF^{2} + \lambda^{-3}
    \cF^{3} \bigr) \cZ^{}_{r,s,\sqrt{n}}}
\]
is either trivial or an eigenmeasure of\/ $\cF$ for the eigenvalue\/
$\lambda$, also in the strict sense.  Further, it is supported
inside\/ $\sqrt{n} \ts \ZZ +F$ for some finite set\/ $F \subset
\RR$. In particular, the support of the measure\/
$\cY^{}_{r,s,\sqrt{n},\lambda} $ is uniformly discrete.
\end{lemma}

\begin{proof}
The eigenmeasure property is clear from
Theorem~\ref{thm:decomp-eigenspaces}.
{Next, set $\theta = \sqrt{n}$. Then, by Lemma~\ref{lem:Z},
we have
\[
\begin{split}
  \cY^{}_{r,s,\theta,\lambda} \, & = \,
    \frac{1}{4} \Bigl( \cZ^{}_{r,s,\theta}  +
   \myfrac{\ee^{2\pi\ii r s}}{\lambda \theta} \cZ^{}_{-s,r,\frac{1}{\theta}}
   +  \lambda^{-2} \cZ^{}_{-r,-s,\theta} +
    \myfrac{\ee^{2\pi\ii r s}}{\lambda^3 \theta}
  \cZ^{}_{s,-r,\frac{1}{\theta}} \Bigr) \\[2mm]
  & = \,   \myfrac{1}{4}\cZ^{}_{r,s,\theta}  +
  \myfrac{1}{4 \lambda^{2}} \cZ^{}_{-r,-s,\theta}
  +  \myfrac{\ee^{2\pi\ii r s}}{4 \lambda \theta} \sum_{m=0}^{n-1}
  \bigl( \cZ^{}_{-s,r+\frac{m}{\theta},\theta} + \myfrac{1}{\lambda^{2}}
  \cZ^{}_{s,-r + \frac{m}{\theta},\theta} \bigr).
\end{split}
\]
This implies}
$\supp \bigl( \cY^{}_{r,s,\sqrt{n},\lambda} \bigr) \subseteq \sqrt{n}
\ts \ZZ + \big\{ \pm {\! s}, \pm \ts\ts r + \frac{m}{\sqrt{n}} : 0 \leqslant m
\leqslant n-1 \big\}$, which clearly is a uniformly discrete subset of
$\RR$.
\end{proof}

For fixed $n\in\NN$ and $\lambda\in \{ 1, \ii, -1, -\ii \}$, we now
define a space of tempered measures via
\begin{equation}\label{eq:def-E}
  \cE^{}_{\lambda} (n) \, \defeq \,
  \text{span}^{}_{\CC} \big\{ \cY^{}_{r,s,\sqrt{n},\lambda}
    : r,s \in \RR \big\} .
\end{equation}
Clearly, any element from $\cE^{}_{\lambda} (n)$ is an eigenmeasure of
$\cF$ with eigenvalue $\lambda$ and uniformly discrete
support.\smallskip

Let us return to the general class of tempered measures with uniformly
discrete support. If such a measure $\mu$ is an eigenmeasure of $\cF$,
the support of $\widehat{\mu}$ must be the same, so $\mu$ is doubly
sparse; see \cite{BST,LO1,LO2} and references therein for background
on such measures. This has a strong consequence as follows.

\begin{prop}\label{prop:sparse}
  Let\/ $\mu\ne 0$ be an eigenmeasure of\/ $\cF$ on\/ $\RR$ such
  that\/ $\supp (\mu) \subset \RR$ is uniformly discrete. Then, there
  exist integers\/ $n,k \in\NN$ such that, with\/ $\beta = \sqrt{n}$,
\[
  \mu \, = \sum_{i=1}^{k} c^{}_{i} \,
  \cZ^{}_{r^{}_{\nts i}, s^{}_{i}, \beta}
\]
   holds for suitable\/ $c^{}_{1}, \ldots , c^{}_{k}
   \in\CC$ and\/ $r^{}_{1}, \ldots, r^{}_{k}, s^{}_{1}, \ldots ,s^{}_{k}
   \in\RR$.
\end{prop}

\begin{proof}
  Clearly, since $\supp (\widehat{\mu}) = \supp (\mu)$, the measure
  $\mu$ is doubly sparse, and an application of \cite[Thms.~1 and
  3]{LO1} guarantees the existence of an integer $\kappa\in\NN$ and
  some number $\alpha>0$ such that
  $\mu = \sum_{j=1}^{\kappa} P_j \, \delta^{}_{\alpha \ZZ + s_j}$,
  where the $P_j$ are trigonometric polynomials. Expanding the latter
  and rewriting $\mu$ in terms of our measures $\cZ^{}_{r,s,\alpha}$
  from \eqref{eq:def-Z} gives
\begin{equation}\label{eq:mu-inter}
  \mu \, = \sum_{i=1}^{\ell} c^{}_{i} \,
  \cZ^{}_{r^{}_{\nts i}, s^{}_{i}, \alpha}
\end{equation}
for some $\ell\in\NN$, which might be larger than $\kappa$, and
numbers $c^{}_{1}, \ldots , c^{}_{\ell}\in\CC$ together with
$r^{}_{1}, \ldots, r^{}_{\ell}, s^{}_{1}, \ldots ,s^{}_{\ell} \in\RR$,
which need not be distinct.

With Lemma~\ref{lem:Z}{\ts}(3), we see that
\[
    \supp (\mu) \, \subseteq \, \alpha \ZZ + F_s
    \quad \text{and} \quad
    \supp (\widehat{\mu}) \, \subseteq \,
    \tfrac{1}{\alpha} \ts \ZZ + F_r
\]
with the finite sets $F_s=\{ s^{}_{1}, \ldots , s^{}_{\ell} \}$ and
$F_r = \{ r^{}_{1} , \ldots , r^{}_{\ell} \}$. By
\cite[Lemma~5.5.1]{NS11}, there exists a finite set $G\subset \RR$
such that
$\frac{1}{\alpha} \ts \ZZ + F_r \subseteq \supp (\widehat{\mu}) + G$,
which implies
\[
   \tfrac{1}{\alpha} \ts \ZZ \, \subseteq \,
   \supp ( \widehat{\mu} ) + G - F_r \, = \,
   \supp (\mu) + G - F_r \, \subseteq \,
   \alpha \ZZ + F_s + G - F_r \, = \, \alpha \ZZ + F^{\ts \prime} ,
\]
where $F^{\ts\prime}$ is still a finite set. Consequently, there are
integers $m,n \in \NN$ with $m<n$ and some element
$t\in F^{\ts\prime}$ such that $\frac{m}{\alpha}$ and
$\frac{n}{\alpha}$ both lie in $\alpha \ZZ + t$. This implies
$\frac{n-m}{\alpha} \in \alpha \ZZ$, hence $0\ne n-m\in\alpha^2 \ZZ$,
and $\alpha^2 \in \QQ$.

Now, let $\alpha^2 = \frac{p}{q}$ with $p,q\in\NN$, and set
$\beta = \sqrt{p \ts q\ts }$, so $\beta = \alpha \ts q$ and
$\alpha = \beta/q$.  Invoking the last relation from
Lemma~\ref{lem:Z}, since $\beta^2 = p\ts q$ is an integer, we can
rewrite $\mu$ from \eqref{eq:mu-inter} as
\[
    \mu \, = \sum_{i=1}^{\ell} \, \sum_{j=0}^{q-1}
    c^{}_{i} \ts \cZ^{}_{r^{}_{\nts i}, s^{}_{i}  + j \alpha, \beta} \ts ,
\]
which, after expanding the double sum and relabelling the parameters,
proves the claim.
\end{proof}

Recall that Lemma~\ref{lem:Z}{\ts}(3), applied several times, leads to
the cycle
\[
  \cZ^{}_{r,s,\alpha} \, \xrightarrow{\,\cF\,} \,
  \tfrac{1}{\alpha}\ts \ee^{2 \pi \ii r s}  \cZ^{}_{-s,r,\frac{1}{\alpha}}
  \, \xrightarrow{\,\cF\,} \, \cZ^{}_{-r,-s,\alpha}
  \, \xrightarrow{\,\cF\,} \,
  \tfrac{1}{\alpha} \ts \ee^{2 \pi \ii r s} \cZ^{}_{s,-r,\frac{1}{\alpha}}
  \, \xrightarrow{\,\cF\,} \, \cZ^{}_{r,s,\alpha}
\]
of length at most $4$, where
$\cZ^{}_{-r,-s,\alpha} = I \nts . \ts \cZ^{}_{r,s,\alpha}$. Ignoring
the actual cycle length, which can be $1$, $2$ or $4$, we get the
following general result.

\begin{theorem}\label{thm:main}
  Let\/ $\mu\ne 0$ be a transformable measure on\/ $\RR$ and\/
  $\lambda$ a fourth root of unity. Then, the following properties are
  equivalent.
\begin{enumerate}\itemsep=2pt
\item The measure\/ $\mu$ is an eigenmeasure of\/ $\cF$ for\/
  $\lambda$ with uniformly discrete support, either in the
  distribution or in the strict sense.
\item There are integers\/ $n,k \in \NN$ and numbers\/
  $c^{}_{1}, \ldots , c^{}_{k} \in \CC$ and\/
  $r^{}_{1}, \ldots , r^{}_{k}, s^{}_{1}, \ldots , s^{}_{k} \in \RR$
  such that one has $\mu \, = \, P^{}_{\lambda} (\nu)$
  for\/ $\nu = \sum_{i=1}^{k} c^{}_{i} \, \cZ^{}_{r_i, s_i, \beta}$
  with\/ $\beta = \sqrt{n}$.
\item There is some\/ $n\in\NN$ such that\/
  $\mu\in \cE^{}_{\lambda} (n)$, as defined in Eq.~\eqref{eq:def-E}.
\end{enumerate}
\end{theorem}

\begin{proof}
  By Lemma~\ref{lem:tb}, the two notions of transformability are
  equivalent in this case, as claimed under (1).

$(1) \Rightarrow (2)$: By Proposition~\ref{prop:sparse}, we
know that $\mu$ must be of the form
\[
   \mu \, = \sum_{i=1}^{k} \ts c^{}_{i} \, \cZ^{}_{r_i, s_i, \beta}
\]
for suitable integers $n,k$ and with suitable numbers $r^{}_{i}$,
$s^{}_{j}$ and $c^{}_{\ell}$. Then, since $\mu \in \EE_{\lambda}(X)$,
where $X$ is the space of measures that are transformable, either in
the distribution or in the strict sense, we have
$\mu=P^{}_{\lambda} (\mu)$ by Proposition~\ref{prop:plambda}.  The
claim follows by setting $\nu=\mu$. \vspace*{1mm}

$(2) \Rightarrow (3)$:  Since $\nu = \sum_{i=1}^{k} c^{}_{i}
\, \cZ^{}_{r_i, s_i, \sqrt{n}}$, we get
\[
  \mu \, = \, P^{}_{\lambda} (\nu) \, =
  \sum_{i=1}^{k} c^{}_{i} \, P^{}_{\lambda}
  \bigl( \cZ^{}_{r_i, s_i, \sqrt{n}} \bigr) \, =
  \sum_{i=1}^{k} c^{}_{i} \, \cY^{}_{r_i, s_i, \sqrt{n}, \lambda}
  \, \in \, \cE^{}_{\lambda} (n) \,.
\]

The implication $(3) \Rightarrow (1)$ is clear from
Lemma~\ref{lem:cycle}.
\end{proof}

For concrete calculations, {given an eigenmeasure $\mu$
  with uniformly discrete support, we can decompose $\mu$ into its
  even and odd part relative to $I$ by
\[
  \mu \, = \, \frac{ \mu + I \nts . \ts \mu }{2}
   + \frac{ \mu - I \nts . \ts \mu }{2} \, \eqdef \,
  \mu^{}_{+} \nts + \mu^{}_{-} \ts .
\]
Then, one derives the following conditions} for
$\widehat{\mu} = \lambda \mu$,
\begin{equation}\label{eq:conditions}
\begin{split}
  \lambda = \pm 1 : & \quad \mu^{}_{-} = 0
  \quad \text{and} \quad
  \widehat{\mu^{}_{+}} = \pm \ts \mu^{}_{+} \ts , \\
  \lambda = \pm\ts\ts \ii\ts : & \quad \mu^{}_{+} = 0
  \quad \text{and} \quad
  \widehat{\mu^{}_{-}} = \pm \ts \mu^{}_{-} \ts ,
\end{split}
\end{equation}
which formed the basis for Lemma~\ref{lem:cycle}.  It is clear from
Proposition~\ref{prop:sparse} in conjunction with
Proposition~\ref{prop:FD} that it suffices to consider measures
\begin{equation}\label{eq:mu-cand}
  \mu \, = \sum_{i=1}^{k} c^{}_{i}\ts
  \cZ^{}_{r^{}_{\nts i}, s^{}_{i}, \alpha}
\end{equation}
with fixed $\alpha=\sqrt{n}$ and parameters
$r^{}_{\nts i} \in \big( \frac{-1}{2\alpha}, \frac{1}{2\alpha} \big]
\eqdef J_r$ and
$s^{}_{i} \in \big( \frac{-\alpha}{2}, \frac{\alpha}{2} \big] \eqdef
J_s$.  Invoking Lemma~\ref{lem:Z}{\ts}(4), we get
$\mu = \mu^{}_{+} \nts + \mu^{}_{-}$ with
\[
  \mu^{}_{\pm} \, = \: \myfrac{1}{2} \sum_{i=1}^{k}
  c^{}_{i} \bigl( \cZ^{}_{r_i, s_i, \alpha}
        \pm \cZ^{}_{-r_i, -s_i, \alpha}\bigr),
\]
where $-r_i$ and $-s_i$ are taken modulo $\frac{1}{\alpha}$ and
$\alpha$, respectively, when $r_i = \frac{1}{2\alpha}$ or
$s_i = \frac{\alpha}{2}$. Now, one can use the fundamental domains to
rewrite the possible eigenmeasures with restricted parameters, and
without ambiguities. Some care has to be exercised with the
$2$-division points
\[
   \big\{ \bigl( 0,0 \bigr), \bigl( 0,\tfrac{\alpha}{2} \bigr),
   \bigl(\tfrac{1}{2\alpha},0 \bigr),
   \bigl(\tfrac{1}{2\alpha},\tfrac{\alpha}{2}\bigr) \big\}
   \,\subset \, J 
\]
under the inversion $I$ in such calculations, the details of which are
left to the interested reader.

\section{Eigenmeasures with locally finite support}

Following \cite{BST}, we say that a measure is \emph{sparse} if its
support is locally finite (that is, discrete and closed). We say that
$\mu$ is \emph{doubly sparse}\footnote{Meyer calls such a measure
  \emph{crystalline} \cite{Meyer}, {which we prefer to
    avoid because `crystals' are often connected with atomic
    arrangements in solids that are uniformly discrete.}} if $\mu$ is
a transformable tempered measure such that both $\mu$ and
$\widehat{\mu}$ are sparse. Finally, when $\mu$ is a doubly sparse
measure which is an eigenmeasure for $\cF$, we will simply call it a
\emph{doubly sparse eigenmeasure}.

Clearly, linear combinations of measures
$\mu^{}_i \in \cE^{}_{\lambda} (n^{}_{i})$, with fixed $\lambda$, are
eigenmeasures with discrete support. The latter is uniformly discrete
if and only if, for all $i\ne j$, one has $\sqrt{m_i/m_j} \in \QQ$,
which is to say that there are many more possibilities. In particular,
for each fourth root of unity, there are doubly sparse eigenmeasures
with non-uniformly discrete support, for instance of the type studied
by Meyer in \cite{Meyer,Meyer-17}.

Notice that, since the union of finitely many locally finite sets is
locally finite, the set $X$ of doubly sparse measures on $\RR^d$ is a
subspace of $\cM_{\text{ttm}}(\RR^d)$, which is by definition
$\cF\nts$-invariant. Therefore, Proposition~\ref{prop:plambda}{\ts}(4)
yields the following result.

\begin{theorem}\label{thm:sparse}
  Let\/ $\mu$ be a transformable tempered measure. Then, $\mu$ is a
  doubly sparse measure if and only if\/ $P^{}_{\lambda} (\mu)$ is a
  doubly sparse eigenmeasure for all\/
  $\lambda \in \{ \pm 1, \pm\ts \ii \}$.

  In particular, a measure is a doubly sparse measure if and only if
  it is a linear combination of doubly sparse eigenmeasures.  \qed
\end{theorem}

Next, we are going to construct doubly sparse eigenmeasures that are
supported inside a lattice, but vanish on arbitrarily large intervals.

Let us first fix $n\in \NN$ and
$\lambda \in \{ \pm 1, \pm \ts \ii \}$. Then, we know from the
multiplicities in Table~\ref{tab:eigen} (which is due to Gauss) that
the dimension $m = \dim (E_{\lambda})$ of the eigenspace $E_{\lambda}$
of $U_n$ satisfies
\[
      4 m + 4 \, \geqslant \, n \ts .
\]
Therefore, we get the following useful result.

\begin{fact}\label{fact:dim-eigen}
    For each\/ $n \in \NN$ and\/  $\lambda \in \{ \pm 1, \pm \ts \ii \}$,
    we have\/ $ \dim (E_{\lambda})  \geqslant
    \max \bigl( 0 \ts , \frac{n}{4} -1  \bigr)$.  \qed
\end{fact}

This gives us the following simple, but powerful lemma.

\begin{lemma}\label{lem:eigenmeasure-zero}
  Let\/ $12\leqslant n\in\NN$ and\/
  $\lambda\in \{ \pm 1 , \pm \ts \ii \}$ be fixed. Let\/ $k \in \NN$
  be chosen such that\/ $1 \leqslant k \leqslant \frac{n}{4}-2$. Then,
  for each\/
  $0 \leqslant j^{}_1 < j^{}_2 < \ldots < j^{}_k \leqslant n-1$, there
  exists some measure\/ $\mu$ that is\/ $\sqrt{n}$-periodic, supported
  inside\/ $\ZZ/\sqrt{n}$, and satisfies\/ $\widehat{\mu}=\lambda \mu$
  together with the vanishing condition\/
  $\mu \bigl( \big\{ j^{}_{\ell} \ts /\nts\sqrt{n} \ts \big\}
  \bigr)=0$ for all\/ $1 \leqslant \ell \leqslant k$.
\end{lemma}

\begin{proof}
Let $X$ be the space of measures that are $\sqrt{n}$-periodic,
supported inside $\ZZ/\sqrt{n}$, and let
\begin{align*}
  Y \,& :=\, \{ \mu \in X: \widehat{\mu}=\lambda \mu \} \ts ,  \\
  Z\, &:=\, \{ \mu \in X: \mu \bigl( \big\{ j^{}_{\ell} \ts/\nts
        \sqrt{n} \ts \big\} \bigr) \, = \, 0 \, \text{ for all }
        1 \leqslant \ell \leqslant k \} \,.
\end{align*}
By Theorem~\ref{thm:lattice}, we have
\[
   \dim(Y) \, = \, \dim(E_\lambda) \, \geqslant \,
    \myfrac{n}{4}-1 \, \geqslant \, k+1 \ts .
\]
On the other hand, we have $ \dim(Z)=n-k $, which shows that
\[
    \dim(Y)+\dim(Z) \,  > \, \dim(X) \ts ,
\]
and hence $Y \cap Z \neq \varnothing$. This implies our claim.
\end{proof}

When $j^{}_1 , j^{}_2 , \ldots , j^{}_k$ are
consecutive numbers, we get the following consequence.

\begin{coro}\label{coro:gap}
  Let\/ $12\leqslant n\in\NN$ and\/
  $\lambda\in \{ \pm 1 , \pm \ts \ii \}$ be fixed, and let\/ \/
  $1 \leqslant k \leqslant \frac{n}{4}-3$.  Then, for each\/
  $0 \leqslant j \leqslant n-k-1$, there exists a non-zero measure\/
  $\mu$ of the form given in Theorem~\textnormal{\ref{thm:lattice}}
  that satisfies\/ $\ts \widehat{\mu} = \lambda \mu$ together with\/
  $\supp (\mu) \subset \ZZ/\sqrt{n}$ and the vanishing condition\/
  $\mu \bigl( \big\{ m/\nts \sqrt{n} \ts \big\} \bigr)=0$ for all\/
  $j \leqslant m \leqslant j+k$.

  Further, for the case\/ $j=0$, the vanishing condition can be
  improved to\/
  $\mu \bigl( \big\{ \ell/\nts \sqrt{n} \ts \big\} \bigr)=0$ for all\/
  $\lvert  \ell \rvert \leqslant k$.
\end{coro}

\begin{proof}
  The first claim is clear.  Now, choosing $j=0$ gives the existence
  of an eigenmeasure $\mu$ that satisfies
  $\ts \widehat{\mu} = \lambda \mu$ together with
  $\supp (\mu) \subset \ZZ/\sqrt{n}$ and
  $\mu \bigl( \big\{ m/\nts \sqrt{n} \ts \big\} \bigr)=0$ for all
  $0 \leqslant m \leqslant k$.

  Finally, since $\mu$ is an eigenmeasure, Fact~\ref{fact:symm} gives
\[
  \mu \bigl( \big\{ -m/\nts \sqrt{n} \ts \big\} \bigr)
  \, = \,  \lambda^2 \mu \bigl( \big\{
  m/\nts \sqrt{n} \ts \big\} \bigr) \, = \, 0 \ts ,
\]
  which holds for all $0 \leqslant m \leqslant k$.
\end{proof}

This shows that doubly sparse measures with large gaps around $0$
exist that are Fourier eigenmeasures.  For $12\leqslant n\in\NN$,
setting $k=\lfloor \frac{n}{4} \rfloor-3$, such measures $\mu$ vanish
on the open interval\/
$\big(-\frac{k+1}{\sqrt{n}} , \frac{k+1}{\sqrt{n}}\big)$, so $\mu$ has
a gap of length\/ $\frac{k+1}{\sqrt{n}}$. Since
\[
  \frac{\sqrt{n}}{4} \, \geqslant \,
  \frac{\lfloor \frac{n}{4} \rfloor}{\sqrt{n}}
  \, \geqslant \,  \frac{\lfloor \frac{n}{4}
     \rfloor -2 }{\sqrt{n}} \, = \,
  \frac{k+1}{\sqrt{n}}
   \, \geqslant \, \frac{\frac{n}{4} -3}{\sqrt{n}}
    \, = \, \frac{\sqrt{n}}{4} - \frac{3}{\sqrt{n}}  \ts ,
\]
the total length of the gap grows like\/ $\frac{\sqrt{n}}{2}$ as
$n \to\infty$.

Let us quickly look at an example for
$\lambda \in \{\pm1,\pm \ts \ii\}$ arbitrary, but fixed. For every
sufficiently large\/ $n\in\NN$, select a non-zero measure
\begin{equation}\label{eq:mu-n}
  \mu^{}_n \, = \, \delta_{\nts\sqrt{n}\ts\ZZ}  * \!
  \sum_{m=0}^{n-1} c^{(n)}_{m} \, \delta^{}_{m/\nts\sqrt{n}}
\end{equation}
where $c^{(n)}_{m}\leqslant 1$ together with
$c^{(n)}_0 = c^{(n)}_1 = \ldots =c^{(n)}_k = 0$ and\/
$c_{n-k} =\ldots=c_{n-1}=0$ constitute an eigenvector of the
corresponding DFT. Such a measure exists by Corollary~\ref{coro:gap}
when $k=\lfloor \frac{n}{4} \rfloor - 3$, and is translation bounded
by construction. Next, such measures can be combined as follows.

\begin{theorem}\label{thm:Meyer1}
  Let\/ $(k^{}_n)^{}_{n\in\NN}$ be a sequence in\/ $\NN$ with\/
  $\lim_n k_n =\infty$. Let\/ $r\in \NN$ be fixed and let\/
  $(a^{}_n)^{}_{ \{n\geqslant 12 \} }$ be a sequence in\/
  $\CC$ such that\/ $|a^{}_n| \leqslant (k_n^{})^r$ holds for all
  sufficiently large\/ $ n \in \NN$.  Then, with the measures\/
  $\mu_n$ from \eqref{eq:mu-n},
\[
    \mu \, \defeq \sum_{n=12}^{\infty} a^{}_n \ts \mu^{}_{k_n}
\]
  is a Radon measure with locally finite support that is tempered
  and satisfies\/ $\widehat{\mu} = \lambda \mu$.
\end{theorem}

\begin{proof}
  Since each $\mu^{}_{k_n}$ vanishes on a gap around $0$ that
  asymptotically grows like $\frac{\sqrt{k_n}}{2}$, the evaluation of
  $\mu$ on a test function of compact support effectively means the
  evaluation of a finite sum.  This shows vague convergence.

  When the test function is from the Schwartz class, the growth
  restriction on the $a_n$ guarantees that the sum also converges in
  the topology of Schwartz space, so $\mu$ is a tempered distribution,
  and thus a tempered measure.

  Since transformability of $\mu$ as a tempered distribution is clear,
  which bypasses the discontinuity issue discussed in \cite{SS}, the
  eigenmeasure property follows by construction.
\end{proof}

This result constructively shows that lots of Fourier eigenmeasures
with locally finite support exist that fail to be translation bounded,
and also fail to be uniformly discrete. Let us formulate two obvious
consequences as follows.

\begin{coro}\label{coro:ev-gap}
  Let\/ $12\leqslant n\in\NN$ and\/
  $\lambda\in \{ \pm 1 , \pm \ts \ii \}$ be fixed. Select some\/
  $0 \leqslant j \leqslant n-k-1$, where\/
  $1 \leqslant k \leqslant \frac{n}{4} -3$. Then, there exists a
  non-zero vector\/
  $c = (c^{}_{0}, c^{}_{1}, \ldots , c^{}_{n-1} )^T \in
  E^{}_{\lambda}$ such that\/ $c_{j+\ell}=0$ holds for all\/
  $\ts 0 \leqslant \ell \leqslant k$.

  In particular, for\/ $j=0$, the symmetry constraints then give\/
  $c^{}_{0} = c^{}_1= \ldots= c^{}_k=0$ together with\/
  $c^{}_{n-1}= c^{}_{n-2}=\ldots= c^{}_{n-k} = 0$.  \qed
\end{coro}

Let us briefly discuss how one can explicitly construct such an
eigenvector, and hence an eigenmeasure as in the second claim of
Corollary~\ref{coro:gap}.

\begin{remark}
  Let $n \geqslant 12$ and let $\{v^{}_1, \ldots, v^{}_m \}$ be a
  basis for the eigenspace $E_{\lambda}$. Let $A$ be the matrix with
  vectors $v_1^T,\ldots, v_m^T$ as rows and
  $c^T = (c^{}_0, \ldots, c^{}_n)$ be the last row in the reduced row
  echelon form of $A$.

  Then, $c\in E_\lambda$ with $c^{}_0 = \ldots c^{}_{m-2}=0$, and
  Remark~\ref{rem:symm} gives $c^{}_{n-1}= \ldots = c^{}_{n-m+2}=0$.
  Consequently, the eigenvector $c$ satisfies the second
  conclusion of Corollary~\ref{coro:ev-gap}.  Then, the corresponding
  eigenmeasure given by \eqref{eq:mu-DFT} is an eigenmeasure with a
  large gap at $0$.\exend
 \end{remark}

\section{Outlook}
The result of Theorem~\ref{thm:sparse} implies that the
characterisation of all doubly sparse measures, an important open
problem in the area of Aperiodic Order as well as in Harmonic
Analysis, is equivalent to the characterisation of all doubly sparse
eigenmeasures. This is an interesting problem which is made difficult
by the fact that, besides linear combination of eigenmeasures with
uniformly discrete support, this class contains many other examples
\cite{Meyer,Meyer-17,Kol}, for instance with a support of the form
$\{ \pm \sqrt{m + 1/r} : m \in \NN_0 \}$ with $r\in\NN$, which is not
uniformly discrete. Two particularly nice ones derive from the work of
Guinand \cite{Gui}.

Note that all measures in Theorem~\ref{thm:Meyer1} are supported
inside $\bigcup_{n} \ZZ/\nts\sqrt{n}$. One can go beyond this
situation by selecting any bounded sequence $(s^{}_n)^{}_{n\in\NN}$
and defining
\[
    \mu := P_{\lambda} \biggl( \sum_{n}
    a^{}_n \ts \delta^{}_{s_n} \!* \mu^{}_{k_n} \biggr)  .
\]
It is then natural to ask whether the measure in \cite{Meyer}, or
even any Fourier eigenmeasure with locally finite support, can be
obtained this way.

Another interesting question, in a rather different direction, is to
what extent, if any, our results can be used in the spectral theory of
dynamical systems. While the equivalence of pure point dynamical
spectrum and pure point diffraction spectrum is well understood, the
full meaning of the eigenmeasure property seems open.

\medskip

\section*{Appendix: Discrete Fourier transform}

Here, we recall some well-known properties of the DFT, {where we refer
  the reader to the rather informative \textsc{WikipediA} entry on the
  DFT for further details and references.} Given $n\in\NN$, the
unitary Fourier matrix $U_n \in \Mat (n,\CC)$ has matrix entries
$\omega^{k\ell}/\sqrt{n}$ with $\omega = \ee^{-2 \pi \ii/n}$ and
\mbox{$0 \leqslant k,\ell \leqslant n-1$}, where we label the entries
starting with $0$.  Since $n=1$ is trivial, we only consider
$n\geqslant 2$. While $U^{}_{2}$ is an involution, $U_n$ has order $4$
for all $n \geqslant 3$, which follows from the relation
\[
    U^{2}_{n} \, = \, \left( \begin{array}{c|ccc}
    1 & 0 & \cdots & 0  \\ \hline 0 & \boldsymbol{0} & & 1 \\
    \raisebox{2pt}{$\vdots$} & & \iddots & \\
    0 & 1 & & \boldsymbol{0} \end{array} \right) .
\]
The eigenvalues and their multiplicities show a structure modulo $4$,
as already known to Gauss. They are summarised in
Table~\ref{tab:eigen}.

\begin{table}
\begin{tabular}{|c|cccc|}
  \hline
  $n$ & $\lambda=1$ & $\lambda =\ii$
  & $\lambda=-1$ & $\lambda=-\ii$ \\
  \hline
  $4m$   & $m+1$ & $m-1$ & $m$   & $m$ \\
  $4m+1$ & $m+1$ & $m$   & $m$   & $m$ \\
  $4m+2$ & $m+1$ & $m$   & $m+1$ & $m$ \\
  $4m+3$ & $m+1$ & $m$   & $m+1$ & $m+1$ \\
  \hline
\end{tabular}\bigskip
\caption{Eigenvalues of $U_n$ and their
    multiplicities.\label{tab:eigen}}
\end{table}

In line with Proposition~\ref{prop:plambda}, for $\lambda \in
\{ \pm  1,  \pm  \ts \ii \}$, we define
\[
     P_{\lambda}^{(n)} \, \defeq \,\myfrac{1}{4}
     \sum_{\ell=0}^{3} \lambda^{- \ell} U^{\ell}_{n} \ts ,
\]
which are projectors with
$P^{(n)}_{\lambda} P^{(n)}_{\lambda'} = \delta^{}_{\lambda, \lambda'}
P^{(n)}_{\lambda}$.  They also satisfy
$U^{}_{n} P^{(n)}_{\lambda} = P^{(n)}_{\lambda} U^{}_{n} = \lambda
P^{(n)}_{\lambda}$ together with
$P^{(n)}_{1} + P^{(n)}_{\ii}+P^{(n)}_{-1}+ P^{(n)}_{-\ii} =
\one^{}_{n}$. In particular, the eigenspace of $U_n$ for the
eigenvalue $\lambda$ is then simply
$E^{}_{\lambda} = P^{(n)}_{\lambda} \CC^n$, and the rank of
$P^{(n)}_{\lambda}$ is the dimension of the eigenspace, as given by
Table~\ref{tab:eigen}.

Concretely, we have
\[
  U^{}_2 \, = \, \myfrac{1}{\sqrt{2}}
  \begin{pmatrix} 1 & 1 \\ 1 & -1 \end{pmatrix}
\]
with eigenvalues $\pm 1$ and eigenvectors $(1\pm \sqrt{2}, 1)$.  Since
the DFT matrices are symmetric, left and right eigenvectors map to
each other under transposition. We state them in row form.  The
projectors read
\[
    P^{(2)}_{\pm \ts 1} \, = \, \myfrac{1}{2} \bigl( \one^{}_{2} \pm
    U^{}_{2} \bigr) \, = \, \myfrac{1}{2\sqrt{2}} \begin{pmatrix}
    \sqrt{2} \pm \ts 1 & \pm \ts 1 \\ \pm \ts 1 & \sqrt{2} \mp \ts 1
    \end{pmatrix} \quad \text{and} \quad
    P^{(2)}_{\pm \ii} \, = \, \nix^{}_{2}  \ts ,
\]
where the first rows of the symmetric projectors $P^{(2)}_{\pm \ts 1}$
correspond to the above eigenvectors.

For $n=3$, setting $\omega = \ee^{-2 \pi \ii/3}$, one finds
\[
  U^{}_3 \, = \, \myfrac{1}{\sqrt{3}}
  \begin{pmatrix} 1 & 1 & 1\\
    1 & \omega & \overline{\omega} \\
    1 & \overline{\omega} & \omega \end{pmatrix} ,
\]
with eigenvectors $(1 \pm \sqrt{3}, 1, 1)$ for $\lambda = \pm 1$ and
$(0,-1,1)$ for $\lambda=-\ii$. Here, we have
$P^{(3)}_{\ii} = \nix^{}_{3}$ together with
\[
    P^{(3)}_{\pm \ts 1} \, = \, \myfrac{1}{4} \begin{pmatrix}
      \frac{2 \bigl( 3 \pm \sqrt{3}\, \bigr)}{3} &
      \pm \frac{2}{\sqrt{3}} &
    \pm \frac{2}{\sqrt{3}_{\vphantom{I}}} \\
    \pm \frac{2}{\sqrt{3}_{\vphantom{I}}} &
    \frac{3\mp\sqrt{3}}{3} &  \frac{3\mp\sqrt{3}^{\vphantom{I}}}{3} \\
     \pm \frac{2}{\sqrt{3}} & \frac{3\mp\sqrt{3}^{\vphantom{I}}}{3}  &
     \frac{3\mp\sqrt{3}}{3}  \end{pmatrix} \quad \text{and}
     \quad P^{(3)}_{-\ii} \, = \, \myfrac{1}{2} \begin{pmatrix}
     0 & 0 & 0 \\ 0 & 1 & -1 \\ 0 & -1 & 1 \end{pmatrix}.
\]

For $n=4$, one has
\[
  U^{}_4 \, = \, \myfrac{1}{2}
  \begin{pmatrix} 1 & 1 & 1 & 1 \\
    1 & - \ii & -1 & \ii  \\
    1 & -1 & 1 & -1 \\
    1 & \ii & -1 & -\ii \end{pmatrix} ,
\]
where the eigenspace for $\lambda=1$ is two-dimensional, spanned by
$(1,0,1,0)$ and $(2,1,0,1)$. The other eigenvectors are $(-1,1,1,1)$
for $\lambda = -1$ and $(0,-1,0,1)$ for $\lambda = -\ii$. Here,
  the projectors are $P^{(4)}_{\ii} = \nix^{}_{4}$ together with
\[
   P^{(4)}_{\pm \ts 1} \, = \, \myfrac{1}{4} \begin{pmatrix}
   2 \pm\ts 1 & \pm \ts 1 & \pm \ts 1 & \pm \ts 1 \\
   \pm \ts 1 & 1 & \mp \ts 1 & 1 \\
   \pm \ts  1 & \mp \ts 1 & 2 \pm \ts 1 & \mp \ts 1 \\
   \pm \ts 1 & 1 & \mp \ts 1 & 1 \end{pmatrix} \quad \text{and} \quad
   P^{(4)}_{-\ii} \, = \, \myfrac{1}{2} \begin{pmatrix}
   0 & 0 & 0 & 0 \\ 0 & 1 & 0 & -1 \\ 0 & 0 & 0 & 0 \\
   0 & -1 & 0 & 1 \end{pmatrix} ,
\]
where the rank of $P^{(4)}_{1}$ is two, as it must.

Next, $n=5$ is the first case where all possible eigenvalues occur.
With $\omega = \ee^{-2 \pi \ii/5}$, we get
\[
  U^{}_5 \, = \, \myfrac{1}{\sqrt{5}}
  \begin{pmatrix} 1 & 1 & 1 & 1 & 1 \\
    1 & \omega & \omega^2 & \overline{\omega}^2 & \overline{\omega} \\
    1 & \omega^2 & \overline{\omega} & \omega & \overline{\omega}^2 \\
    1 & \overline{\omega}^2 & \omega & \overline{\omega} & \omega^2 \\
    1 & \overline{\omega} & \overline{\omega}^2 & \omega^2 & \omega
  \end{pmatrix} ,
\]
where $(\tau,1,0,0,1)$ and $(\tau,0,1,1,0)$ span the eigenspace of
$\lambda=1$, with $\tau = \frac{1}{2} (1+\sqrt{5}\,)$ being the golden
ratio. The other eigenvectors are
$(0,-1,\tau\ts {\pm} \ts\eta,-(\tau \ts {\pm} \ts \eta),1)$ with
$\eta = \sqrt{\tau {+} 2 \ts}$ for $\lambda = \pm \ts \ii$, and
$(-\frac{2}{\tau},1,1,1,1)$ for $\lambda = -1$. Here, the projectors
are
\[
  P^{(5)}_{\pm 1} \, = \, \myfrac{1}{20} \begin{pmatrix}
   10 \pm (4 \tau {-} 2) & \pm (4 \tau {-} 2) &
   \pm ( 4 \tau {-} 2) & \pm ( 4 \tau {-} 2) &
   \pm ( 4 \tau {-} 2)  \\
   \pm ( 4 \tau {-} 2) & 5 \pm (3 {-} \tau) & \mp (\tau {+} 2) &
   \mp (\tau {+} 2) & 5 \pm (3 {-} \tau) \\
   \pm ( 4 \tau {-} 2) & \mp (\tau {+} 2) & 5 \pm (3 {-} \tau) &
   5 \pm (3 {-} \tau) & \mp (\tau {+} 2) \\
   \pm ( 4 \tau {-} 2) & \mp (\tau {+} 2) & 5 \pm (3 {-} \tau) &
   5 \pm (3 {-} \tau) & \mp (\tau {+} 2) \\
   \pm ( 4 \tau {-} 2) & 5 \pm (3 {-} \tau) & \mp (\tau {+} 2) &
   \mp (\tau {+} 2) & 5 \pm (3 {-} \tau) \end{pmatrix}
\]
and
\[
  P^{(5)}_{\pm\ts\ii} \, = \, \myfrac{1}{40} \begin{pmatrix}
    0 & 0 & 0 & 0 & 0 \\
    0 & 10 \mp \nts \sqrt{20}\, \eta & \mp \nts \sqrt{60{-}20\ts\tau} &
    \pm \nts \sqrt{60{-}20\ts\tau} & -10 \pm \nts \sqrt{20} \, \eta \\
    0 & \mp \nts \sqrt{60{-}20\ts\tau} & 10 \pm \nts \sqrt{20} \, \eta &
    -10 \mp \nts \sqrt{20} \, \eta & \pm \nts \sqrt{60{-}20\ts\tau} \\
    0 & \pm \nts \sqrt{60{-}20\ts\tau} & -10 \mp \nts \sqrt{20} \, \eta &
    10 \pm \nts \sqrt{20} \, \eta &  \mp \nts \sqrt{60{-}20\ts\tau} \\
    0 & -10 \pm \nts \sqrt{20} \, \eta & \pm \nts \sqrt{60{-}20\ts\tau} &
    \mp \nts \sqrt{60{-}20\ts\tau} & 10 \mp \nts \sqrt{20} \, \eta
    \end{pmatrix} ,
\]
where the rank of $P^{(5)}_{1}$ is two, while all others have rank one.

As a last case, let us consider $n=6$. Here, with $\omega =
\ee^{-\pi \ii/3}$, we have
\[
   U^{}_{6} \, = \, \myfrac{1}{\sqrt{6}}
  \begin{pmatrix} 1 & 1 & 1 & 1 & 1 & 1 \\
    1 & \omega & \omega^2 & -1 &
        \overline{\omega}^2 & \overline{\omega} \\
    1 & \omega^2 & \overline{\omega}^2 & 1 &
        \omega^2 & \overline{\omega}^2 \\
    1 & -1 & 1 & -1 & 1 & -1 \\
    1 & \overline{\omega}^2 & \omega^2 & 1&
        \overline{\omega}^2 & \omega^2 \\
    1 & \overline{\omega} & \overline{\omega}^2 &
        -1 & \omega^2 & \omega
  \end{pmatrix} .
\]
The eigenspace for $\lambda=1$ is two-dimensional, spanned by
$(\beta,1,0,1-\beta,0,1)$ together with $(1+\beta,0,1,\beta,1,0)$, as
is that for $\lambda = -1$, which is spanned by
$(-\beta,1,0,1+\beta,0,1)$ and $(1-\beta,0,1,-\beta,1,0)$, with
$\beta = \sqrt{3/2}$ in both cases. The remaining eigenvectors are
given by $(0,-1,1\pm\sqrt{2},0, -(1\pm\sqrt{2}\,),1)$ for
$\lambda = \pm \ts \ii$, respectively. The projectors read
\[
  P^{(6)}_{\pm 1} \, = \, \myfrac{1}{4 \sqrt{6}} \begin{pmatrix}
    2 ( \sqrt{6} \pm \ts 1) & \pm \ts 2 & \pm \ts 2 & \pm \ts 2
    & \pm \ts 2 & \pm \ts 2 \\ \pm \ts 2 & \sqrt{6} \pm \ts 1 &
    \mp \ts 1 & \mp \ts 2 & \mp \ts 1 & \sqrt{6} \pm \ts 1 \\
    \pm \ts 2 & \mp \ts 1 & \sqrt{6} \mp \ts 1 & \pm \ts 2 &
    \sqrt{6} \mp \ts 1 & \mp \ts 1 \\ \pm \ts 2 & \mp \ts 2 &
    \pm \ts 2 & 2(\sqrt{6} \mp \ts 1) & \pm \ts 2 & \mp \ts 2 \\
    \pm \ts 2 & \mp \ts 1 & \sqrt{6} \mp \ts 1 & \pm \ts 2 &
    \sqrt{6} \mp \ts 1 & \mp \ts 1 \\ \pm \ts 2 &
    \sqrt{6} \pm \ts 1 & \mp \ts 1 & \mp \ts 2 & \mp \ts 1 &
    \sqrt{6} \pm \ts 1 \end{pmatrix}
\]
and
\[
 P^{(6)}_{\pm \ts \ii} \, = \, \myfrac{1}{8} \begin{pmatrix}
    0 & 0 & 0 & 0 & 0 & 0 \\
    0 & 2 \mp \sqrt{2} & \mp \sqrt{2} &
    0 & \pm \sqrt{2} & - 2 \pm \sqrt{2} \\
    0 & \mp \sqrt{2} & 2 \pm \sqrt{2} &
    0 & -2 \mp \sqrt{2} & \pm \sqrt{2} \\
    0 & 0 & 0 & 0 & 0 & 0 \\
    0 & \pm \sqrt{2} & -2 \mp \sqrt{2} &
    0 & 2 \pm \sqrt{2} & \mp \sqrt{2} \\
    0 & -2 \pm \sqrt{2}& \pm \sqrt{2} &
    0 & \mp \sqrt{2} & 2 \mp \sqrt{2} \end{pmatrix} .
\]
Here, $P^{(6)}_{\pm 1}$ have rank two and $P^{(6)}_{\pm\ts\ii}$
are of rank one.

\section*{Acknowledgements}

We are grateful to {Hans Georg Feichtinger}, Sebastian Herr, Yves
Meyer, Adi Tcaciuc and Venta Terauds for valuable discussions and
helpful comments on the manuscript. We thank the referees for their
constructive comments which helped us to improve the presentation. TS
would like to thank the Department of Mathematical and Statistical
Sciences of the University of Alberta (Edmonton, Canada) for
hospitality during his stay as a research fellow of the DFG.  This
work was supported by the German Research Foundation (DFG, Deutsche
Forschungsgemeinschaft), within the CRC 1283/2 \mbox{(2021 -
  317210226)} at Bielefeld University (MB) and via research grant
415818660 (TS), and by the Natural Sciences and Engineering Council of
Canada (NSERC), via grant 2020-00038 (NS).

\end{document}